\documentclass[reqno]{amsart}

\usepackage{amsmath,amssymb,amsthm}
\usepackage{mathtools}
\usepackage{here}
\usepackage{multirow}
\usepackage{url}
\usepackage[noadjust]{cite}
\usepackage{enumerate}
\usepackage[driverfallback=dvipdfm]{hyperref}
\hypersetup{
colorlinks=true,
linkcolor=black,
citecolor=black
}

\theoremstyle{definition}
\newtheorem{thm}{Theorem}
\newtheorem{de}{Definition}
\newtheorem{prop}{Proposition}
\newtheorem{lem}{Lemma}
\newtheorem{cor}{Corollary}
\newtheorem{rem}{Remark}
\newtheorem{ex}{Example}

\begin{document}

\title{On the orthogonality of Atkin-like polynomials and orthogonal polynomial expansion of generalized Faber polynomials}

\author{Tomoaki Nakaya}

\address{Faculty of Mathematics, Kyushu University, 744, Motooka, Nishi-ku, Fukuoka, 819-0395, Japan}
\email{t-nakaya@math.kyushu-u.ac.jp}

\begin{abstract}
In this paper, we consider the Atkin-like polynomials that appeared in the study of normalized extremal quasimodular forms of depth 1 on $SL_{2}(\mathbb{Z})$ by Kaneko and Koike as orthogonal polynomials and clarify their properties. By considering Atkin-like polynomials in terms of orthogonal polynomials, we prove an unexpected connection between generalized Faber polynomials, which are closely related to certain bases of the vector space of weakly holomorphic modular forms, and normalized extremal quasimodular forms. In particular, we show that the orthogonal polynomial expansion coefficients of the generalized Faber polynomials by the Atkin-like polynomials appear in the Fourier coefficients of normalized extremal quasimodular forms multiplied by certain (weakly) holomorphic modular forms.
\end{abstract}

\maketitle

%%%%%%%%%%%%%%%%%%%%%%%%%%%%%%%%%%%%%%%%%%%%%%%%%%%%%%%%%%%%%%%%%%%%%%%%%%%%%%%%%%%%%%%%%
%%%%%%%%%%%%%%%%%%%%%%%%%%%%%%%%%%%%%%%%%%%%%%%%%%%%%%%%%%%%%%%%%%%%%%%%%%%%%%%%%%%%%%%%%
%%%%%%%%%%%%%%%%%%%%%%%%%%%%%%%%%%%%%%%%%%%%%%%%%%%%%%%%%%%%%%%%%%%%%%%%%%%%%%%%%%%%%%%%%

\section{Introduction and Preliminaries} \label{sec:intro}

Let $\mathfrak{H}$ be the complex upper half plane and 
$\tau \in \mathfrak{H},\, q=e^{2\pi i \tau}$ and $\Gamma=SL_{2}(\mathbb{Z})$. 
We denote by $\mathcal{M}$ the $\mathbb{C}$-vector space of the  holomorphic $\Gamma$-invariant function in $\mathfrak{H}$ which is meromorphic at $i \infty$. Then the space $\mathcal{M}$ coincides with the polynomial ring $\mathbb{C}[j]$, where $j=j(\tau)=q^{-1} + 744 +196884 q + \cdots$ is the elliptic modular invariant. 
In the 1980s, inspired by Rankin's work on the zeros of the Eisenstein series \cite{rankin1968zerosEisenstein}, Atkin defined the following inner product $(\;,\;)$ on $\mathcal{M}$:
 \begin{align}
(f,g) = \text{constant term of $fgE_{2}$ as a Laurent series in $q$}, \label{def:Atkininnerprod}
\end{align}
where $E_{2}$ is the normalized quasimodular Eisenstein series of weight 2 on $\Gamma$. We call this inner product the Atkin inner product. 
This inner product has the following remarkable properties. 
\begin{enumerate}[(i)]
\item The Atkin inner product is non-degenerate and positive definite on $\mathbb{R}[j]$, and the Hecke operators $T_{n}:\mathcal{M} \to \mathcal{M}\,(n \in \mathbb{N})$ are self-adjoint with respect to this inner product, i.e., $(f|_{0}T_{n},g)=(f,g|_{0}T_{n})$ holds for $f,g \in \mathcal{M}$. 
An inner product on $\mathcal{M}$ with these properties is unique up to  a scalar multiple.
\item Let $\{A_{n}(j)\}_{n\ge0}$ be an orthogonal polynomial system with respect to the Atkin inner product. We assume that $A_{n}(j)$ is a monic polynomial of degree $n$. Then we have $A_{n}(j) \in \mathbb{Q}[j]$. 
Furthermore, if $p$ is a prime and $n_{p}= \deg ss_{p}(j)$, then $A_{n_{p}}(j)$ has $p$-integral coefficients and $A_{n_{p}}(j) \equiv ss_{p}(j) \pmod{p}$ holds, where the polynomial $ss_{p}(j)$ is the supersingular polynomial of characteristic $p$. In other words, the roots of $A_{n_{p}}(j) \pmod{p}$ give the $j$-invariants of supersingular elliptic curves in characteristic $p$.
\end{enumerate}
We call this polynomial $A_{n}(X)$ the Atkin's orthogonal polynomial or Atkin polynomial for short. Atkin did not publish these results, but we can see a simplified proof of his discoveries in \cite{kaneko1998supersingular} by Kaneko and Zagier. They also give recursion and closed formula for Atkin polynomials.

Later, in \cite{kaneko2006extremal}, Kaneko and Koike defined extremal quasimodular forms of even weight and depth $\ge1$ on $\Gamma$ and proved that those of depth 1 are solutions of the differential equation (or its analogue) which appeared in \cite{kaneko1998supersingular}. In particular, Atkin polynomials $A_{n}(X)$ appear in the explicit representation of the extremal quasimodular forms of weight $12n+2$ and depth 1 on~$\Gamma$. 
They also introduce a polynomial sequence $\{A_{n,r}(X)\}_{n\ge0}$ related to the extremal quasimodular forms of weight $12n+r \; (r \in \{0,6,8\})$, which satisfies a recursion similar to the Atkin polynomials. (Throughout this paper we will refer to the polynomials $A_{n,r}(X)$ collectively as the Atkin-like polynomials.) 
However, neither closed formulas for these polynomials nor their aspect as orthogonal polynomials are not given in \cite{kaneko2006extremal}. 
The purpose of this paper is to show them and to provide a new perspective on normalized extremal quasimodular forms as an application.

The paper is organized as follows. 
In Section \ref{sec:clfAtkinpoly}, we give closed formulas for the Atkin-like polynomials $A_{n,r}(X)$ and their adjoint polynomials $B_{n,r}(X)$. 
The method of proof is the same as for the Atkin polynomial in \cite{kaneko1998supersingular}, but for the first time we give a closed formula for the polynomials $B_{n,r}(X)$, which has been ignored so far.

Section \ref{sec:orthogonalityA} clarifies the aspects of the Atkin-like polynomials $A_{n,r}(X)$ as orthogonal polynomials.
The key to the proof are various formulas for the linear functional $\mathcal{L}$ corresponding to the Atkin inner products, which follow from the results of the previous section that the polynomials $A_{n,r}(X)$ (with some factors) can be expanded in certain truncated hypergeometric series. 
In particular, the fact that the normalized extremal quasimodular forms have an expression using $\mathcal{L}$ leads to a relationship between the polynomials $A_{n,r}(X)$ and the generalized Farber polynomials, which will be pursued in the next section.

In Section \ref{sec:genFaberpoly} we introduce the generalized Faber polynomials $F_{k,n}(X)$ of weight $k$ and degree~$n$ on $\Gamma$, which are closely related to certain bases of the vector space of weakly holomorphic modular forms on $\Gamma$. Using a result of Duke and Jenkins on the generating function of the generalized Faber polynomials, we will prove that the coefficients of the orthogonal polynomial expansion of the generalized Faber polynomials using the Atkin-like polynomials $A_{n,r}(X)$ are the Fourier coefficients of the normalized extremal quasimodular forms multiplied by the appropriate powers of $E_{4}, E_{6}$, and $\Delta$.

In Section \ref{sec:Atkin-Petersson} we show that the Atkin inner product can be interpreted as the weight zero counterpart of the Petersson inner product. 
We then use the results of the previous section to give a partial answer to Kaneko's question concerning the Atkin inner products $(H_{n}(j), A_{\ell}(j))$ of the ``Poincar\'{e} series'' $H_{n}(j)$ of weight $0$ and the Atkin polynomials. 
Furthermore, inspired by Duke, Imamo\={g}lu and T\'{o}th's work on regularized Petersson inner products, we also consider generating series of $(H_{n}(j), H_{\ell}(j))$.
A slight modification of this result yields the generating series of the inner products $(F_{k,n}(j), F_{k,\ell}(j))$ of the generalized Faber polynomials, which in fact can be shown to have a simple closed form.

Finally, although of little relevance to the polynomials $A_{n,r}(X)$ that are the main subject of this paper, Section \ref{sec:HypgenFaber} discusses the hypergeometric aspects of the generalized Faber polynomials.
For the proof we will use the classical result of the relation between the inverse series of the inverse of the modular invariant $j(\tau )$ and the quotient of two independent solutions of a certain hypergeometric differential equation.

In the following, we present the symbols and statements used in this paper without proofs from Sections 1 and 2 of the paper \cite{nakaya2023determination}.
A quasimodular form of weight $w$ on $\Gamma$ is given as
\begin{align*}
\sum_{\substack{2\ell+4m+6n=w \\ \ell,m,n \geq 0}} C_{\ell,m,n} E_{2}^{\,\ell} E_{4}^{\,m} E_{6}^{\,n} \in QM_{*}(\Gamma) \coloneqq \mathbb{C}[E_{2},E_{4},E_{6}] . 
\end{align*}
Here $E_{k}=E_{k}(\tau)$ is the standard Eisenstein series on $\Gamma$ of weight $k$ defined by
\begin{align*}
E_{k}(\tau ) = 1 - \frac{2k}{B_{k}} \sum_{n=1}^{\infty} \sigma_{k-1}(n)  q^{n}, \quad \sigma_{k}(n) = \sum_{d \mid n} d^{k}, 
\end{align*}
where $B_{k}$ is $k$-th Bernoulli number, e.g., $B_{2}=\tfrac{1}{6}, B_{4}= -\tfrac{1}{30}, B_{6}= \tfrac{1}{42}$. 
We denote the vector space of modular forms and cusp forms of weight $k$ on $\Gamma$ by $M_{k}(\Gamma)$ and $S_{k}(\Gamma)$, respectively. It is well known that $E_{k} \in M_{k}(\Gamma)$ for even $k\ge4$, but $E_{2}$ is not modular and is quasimodular. 
Specifically, for $\bigl( \begin{smallmatrix}  a & b \\ c & d \end{smallmatrix} \bigr) \in \Gamma$, we have
\begin{align*}
&{} E_{k}\left( \frac{a \tau +b}{c \tau +d} \right) = (c \tau +d)^{k}E_{k}(\tau) \quad (k\ge4), \\
&{} E_{2}\left( \frac{a \tau +b}{c \tau +d} \right) = (c \tau +d)^{2}E_{2}(\tau ) + \frac{6}{\pi i} c(c \tau +d).
\end{align*}
We also define the ``discriminant'' cusp form $\Delta$ of weight $12$ on $\Gamma$ as follows.
\begin{align*}
\Delta (\tau ) = q \prod_{n=1}^{\infty} (1-q^{n})^{24} = \frac{E_{4}(\tau)^{3} - E_{6}(\tau)^{2}}{1728} \in S_{12}(\Gamma) .
\end{align*}
Then the elliptic modular invariant $j(\tau )$ is defined by $j(\tau )=E_{4}(\tau )^{3}/\Delta (\tau )$.

Any quasimodular form $f$ of weight $w$ can be uniquely written as 
\begin{align*}
f = \sum_{\ell=0}^{r} E_{2}^{\,\ell} f_{\ell}, \quad f_{\ell} \in M_{w-2\ell}(\Gamma), \quad  f_{r} \not= 0
\end{align*}
with $r \in \mathbb{Z}_{\ge0}$ and we call it depth of $f$. 
Let $QM_{w}^{(r)} = QM_{w}^{(r)}(\Gamma)$ denote the vector space of quasimodular forms of weight $w$ and depth $\le r$ on $\Gamma$. In particular, $QM_{w}^{(0)}(\Gamma)$ is equal to $M_{w}(\Gamma)$. (Hereafter, we often omit the reference to the group $\Gamma$.) 
From the fact $\dim_{\mathbb{C}} QM_{w}^{(r)} = \sum_{\ell=0}^{r} \dim_{\mathbb{C}} M_{w-2\ell}$, the generating function of the dimension of $QM_{w}^{(r)}$ is given by the following:
\begin{align*}
&\sum_{k=0}^{\infty} \dim_{\mathbb{C}} QM_{2k}^{(r)}\, T^{2k} = \frac{\sum_{\ell=0}^{r} T^{2\ell}}{(1-T^{4})(1-T^{6})} \\
&= \frac{1-T^{2(r+1)}}{(1-T^{2})(1-T^{4})(1-T^{6})} \quad (r \in \mathbb{Z}_{\ge0}, \; |T|<1 ) .
\end{align*}
See \cite{grabner2020quasimodular} for the explicit formula of $\dim_{\mathbb{C}} QM_{w}^{(r)}$.

We denote by $D$ the differential operator $D = \tfrac{1}{2 \pi i } \tfrac{d}{d \tau} = q \tfrac{d}{d q}$. It is well known that the Eisenstein series $E_{2},E_{4}$ and $E_{6}$ satisfy the following differential relation (equation\footnote{See \cite[\S 3]{cruz2018manifold} and \cite[\S 4.1]{movasati2012quasimodular} for Halphen's contribution.}) by Ramanujan \cite{ramanujan2000arithmeticalfunctions} (See also \cite[Thm.~0.21]{cooper2017ramanujan} and notes on p.~43.):
\begin{align}
D(E_{2}) = \frac{E_{2}^{2} - E_{4}}{12}, \quad D(E_{4}) = \frac{E_{2}E_{4} - E_{6}}{3}, \quad D(E_{6}) = \frac{E_{2}E_{6} - E_{4}^{2}}{2} . \label{eq:DEisen}
\end{align}
Therefore, the ring $QM_{*}(\Gamma)$ is closed under the derivation $D$.
Using the infinite product expression of $\Delta(\tau)$ or the differential relation \eqref{eq:DEisen}, we see that $D(\log \Delta) = D(\Delta)/\Delta = E_{2}$.

\begin{de}
We call a quasimodular form $f \in QM_{w}^{(r)} \backslash QM_{w}^{(r-1)}$ is extremal if it is expressed in the form $f = c\, q^{m-1}(1+O(q))$, where $m = \dim_{\mathbb{C}} QM_{w}^{(r)}$ and $c \in \mathbb{C}\backslash\{0\}$. 
If $c=1$, $f$ is said to be normalized. We denote by $G_{w}^{(r)}$ the normalized extremal quasimodular form of weight $w$ and depth $r$ on $\Gamma$ (if it exists). 
\end{de}
As discussed in the first section of \cite{nakaya2023determination}, the normalized extremal quasimodular forms $G_{w}^{(r)}$ exist uniquely for depth $r$ between 1 and 4. 
Since we are only dealing with the case of depth 1 in this paper, we will write $G_{w}^{(1)}$ simply as $G_{w}$.

Here are some examples of $G_{w}$ with small weights.
\begin{align*}
G_{2} &= E_{2} = D(\log \Delta) = 1 -24 \sum_{n=1}^{\infty} \sigma_{1}(n) q^{n}, \\
G_{6} &=  \frac{E_{2}E_{4} - E_{6}}{720} = \frac{D(E_{4})}{240} = \sum_{n=1}^{\infty} n \sigma_{3}(n) q^{n} , \\
G_{8} &= \frac{E_{4}^{2} - E_{2}E_{6}}{1008} = - \frac{D(E_{6})}{504}  = \sum_{n=1}^{\infty} n \sigma_{5}(n) q^{n} , \\ 
G_{10} &= E_{4}G_{6}^{(1)} = \frac{D(E_{8})}{480}  = \sum_{n=1}^{\infty} n \sigma_{7}(n) q^{n} ,\\
G_{12} &= \frac{E_{4}^{3} - 1008 \Delta - E_{2}E_{4}E_{6}}{332640} = \frac{1}{2 \cdot 3 \cdot 5^{2} \cdot 7} \sum_{n=2}^{\infty} ( n\sigma_{9}(n) - \tau(n) ) q^{n}  \\
&= q^{2}+56 q^{3}+1002 q^{4}+9296 q^{5}+57708 q^{6}+269040 q^{7} +O(q^{8}),  \\
G_{14} &= \frac{E_{2} (E_{4}^{3} - 720 \Delta)  - E_{4}^{2}E_{6}}{393120} = \frac{1}{2 \cdot 3 \cdot 691} \sum_{n=2}^{\infty} n(\sigma_{11}(n) - \tau(n) ) q^{n} \\
&=  q^{2}+128 q^{3}+4050 q^{4}+58880 q^{5}+525300 q^{6}+3338496 q^{7} +O(q^{8}) , 
\end{align*}
where Ramanujan's tau-function $\tau (n)$ is defined by $\Delta = \sum_{n=1}^{\infty} \tau (n) q^{n}$. 
In these examples (although non-trivial for the last two) $G_{w} \in \mathbb{Z}[\![q]\!]$ holds, but in fact there are only 22 weights $w$ with such a property. 
See \cite{nakaya2023determination} for details.

Though the symbols are slightly different, in their paper \cite{kaneko2006extremal} Kaneko and Koike expressed the normalized extremal quasimodular forms $G_{w}$ by using the monic polynomials $A_{m,r}(X)$ and $B_{m,r}(X)$ as follows:
\begin{align}
&{} G_{12m} = \frac{1}{N_{m,0}} \left( - E_{2} E_{4} E_{6} \Delta^{m-1} A_{m,0}(j) +  \Delta^{m} B_{m,0}(j)  \right), \label{eq:G12mAB} \\
&{} G_{12m+2} = \frac{1}{N_{m,2}} \left( E_{2} \Delta^{m} A_{m,2}(j) - E_{4}^{2} E_{6} \Delta^{m-1} B_{m,2}(j)  \right), \label{eq:G12m2AB}\\
&{} G_{12m+6} = \frac{1}{N_{m,6}} \left( E_{2} E_{4} \Delta^{m} A_{m,6}(j) - E_{6} \Delta^{m} B_{m,6}(j)  \right), \label{eq:G12m6AB}\\
&{} G_{12m+8} = \frac{1}{N_{m,8}} \left( - E_{2} E_{6} \Delta^{m} A_{m,8}(j) + E_{4}^{2} \Delta^{m} B_{m,8}(j)  \right) ,  \label{eq:G12m8AB}
\end{align}
where the normalizing factor $N_{m,r}$ is given below for any $m\ge0$ except for $N_{0,0}=N_{0,2}\coloneqq1$. 
\begin{align}
N_{m,0} &= 24m \binom{6m}{2m} \binom{12m}{6m} , \quad  N_{m,2} = \frac{12m+1}{12m-1} N_{m,0} , \label{eq:d1nf02} \\
N_{m,6} &= N_{m+1/2,0} = 12(2m+1) \binom{6m+3}{2m+1} \binom{12m+6}{6m+3}, \quad N_{m,8} = N_{m+1/2,2}. \label{eq:d1nf68}
\end{align}
Here, $N_{m,2}$ and $N_{m,8}$ are integers, because $\tfrac{1}{2n-1} \tbinom{2n}{n} = -\tbinom{2n}{n} +4 \tbinom{2n-2}{n-1}$ holds.
In particular, the polynomial $A_{m,2}(X)$ is nothing but the original Atkin polynomial $A_{m}(X)$ treated in \cite{kaneko1998supersingular}.

Let $p,q$ be nonnegative integers and $a_{i},b_{j} \in \mathbb{C}$ with $b_{j}\not\in \mathbb{Z}_{\le0}$. The generalized hypergeometric series ${}_{p}F_{q}$ is defined by
\begin{align*}
{}_{p}F_{q} \left( a_{1}, \dots , a_{p} ; b_{1}, \dots , b_{q} ; z \right) = \sum_{n=0}^{\infty} \frac{(a_{1})_{n} \dotsm (a_{p})_{n}}{(b_{1})_{n} \dotsm (b_{q})_{n}} \frac{z^{n}}{n!}, 
\end{align*}
where $(a)_{0}=1, (a)_{n} = a(a+1) \cdots (a+n-1)\;(n\ge1)$ denotes the Pochhammer symbol. This series is clearly invariant under the interchange of each of the parameters $a_{i}$, and the same is true for $b_{j}$. 
When $p=q+1$, the series $F={}_{p}F_{p-1}\left( a_{1}, \dots , a_{p} ; b_{1}, \dots , b_{q} ; z \right)$ satisfies the differential equation
\begin{align*}
z^{p-1} (1 - z) \frac{d^{p}F}{d z^{p}} + \sum_{n=1}^{p-1} z^{n-1}(\alpha_{n}z + \beta_{n}) \frac{d^{n}F}{d z^{n}} +\alpha_{0}\,F =0,
\end{align*}
where $\alpha_{n}$ and $\beta_{n}$ are some constants depend on the parameters $a_{i}$ and $b_{j}$. 
By using the Euler operator $\Theta = z \tfrac{d}{d z}$, the above differential equation can be rewritten as
\begin{align*}
\{ \Theta (\Theta +b_{1} -1) \dotsm (\Theta +b_{p-1} -1) - z (\Theta + a_{1}) \dotsm (\Theta + a_{p}) \} F =0. 
\end{align*}
The following transformation formula is valid for the Gaussian hypergeometric series ${}_{2}F_{1}$:
\begin{align}
{}_{2}F_{1} \left( \alpha, \beta ; \gamma ; z \right) = (1-z)^{\gamma-\alpha-\beta} {}_{2}F_{1} \left( \gamma -\alpha, \gamma -\beta ; \gamma ; z \right). \label{eq:Euler} 
\end{align}

The Eisenstein series $E_{4},E_{6}$, and $E_{2}$ have the following hypergeometric expressions: see \cite{stiller1988classical} for the first two and \cite{nakaya2023determination} for $E_{2}$.
\begin{prop}\label{prop:EisensteinHyp}
For sufficiently large $\Im (\tau)$, we have
\begin{align*}
E_{4}(\tau ) &= {}_{2}F_{1} \left( \frac{1}{12}, \frac{5}{12} ; 1 ; \frac{1728}{j(\tau)}   \right)^{4}, \\
E_{6}(\tau ) &= \left( 1 - \frac{1728}{j(\tau)} \right)^{1/2} {}_{2}F_{1} \left( \frac{1}{12}, \frac{5}{12} ; 1 ; \frac{1728}{j(\tau)}   \right)^{6} ,  \\ 
E_{2}(\tau)  &=  {}_{2}F_{1} \left( \frac{1}{12}, \frac{5}{12} ; 1 ; \frac{1728}{j(\tau)}   \right) {}_{2}F_{1} \left( -\frac{1}{12}, \frac{7}{12} ; 1 ; \frac{1728}{j(\tau)}   \right) .
\end{align*}
\end{prop}

We introduce the Serre derivative (or Ramanujan--Serre derivative) $\partial_{k}$ defined by
\begin{align*}
\partial_{k} = D - \frac{k}{12} E_{2} .
\end{align*}
From this definition, it is clear that Leibniz rule $\partial_{k+l}(f g)=\partial_{k}(f)g+f\partial_{l}(g)$ is satisfied. We use the following symbols for iterated Serre derivative according to convention:
\begin{align*}
\partial_{k}^{0}(f) = f, \quad \partial_{k}^{n+1}(f) = \partial_{k+2n} \circ \partial_{k}^{n}(f) \quad (n\ge0).
\end{align*}
It is well known that $\partial_{k-r} : QM_{k}^{(r)} \rightarrow QM_{k+2}^{(r)}$ for even $k$ and $r \in \mathbb{Z}_{\ge0}$, which is a special case of Proposition 3.3 in \cite{kaneko2006extremal}.
By rewriting \eqref{eq:DEisen} using the Serre derivative, we obtain the following useful consequences. 
\begin{align}
\partial_{1}(E_{2}) = -\frac{1}{12} E_{4}, \quad \partial_{4}(E_{4}) = -\frac{1}{3} E_{6}, \quad \partial_{6}(E_{6}) = -\frac{1}{2} E_{4}^{2}. \label{eq:exSerred}
\end{align}

Now consider the following differential equation, called the Kaneko--Zagier equation or, in a more general context, a second-order modular linear differential equation: 
\begin{align}
D^{2}(f) - \frac{w}{6} E_{2} D(f) + \frac{w(w-1)}{12} D(E_{2}) f =0. \label{eq:2ndKZeqn}
\end{align}
This differential equation (with $w$ replaced by $k+1$) first appeared in the study of the $j$-invariants of supersingular elliptic curves in \cite{kaneko1998supersingular} and characterized in \cite[\S 5]{kaneko2003modular}. 
As stated in \cite{kaneko2006extremal}, the normalized extremal quasimodular form $G_{6n}$ is the solution of \eqref{eq:2ndKZeqn} for $w=6n$.
On the other hand, by change of variables $z=1728/j(\tau )$ and putting $g(\tau ) = E_{4}(\tau )^{-(w-1)/4} f(\tau )$, the Kaneko--Zagier equation \eqref{eq:2ndKZeqn} is transformed into a hypergeometric differential equation. 
On the basis of these facts, a hypergeometric expression of $G_{k}$ was given in \cite{nakaya2023determination} as follows.

\begin{prop} \label{prop:Gto2F1}
The normalized extremal quasimodular forms of even weight and depth $1$ on $\Gamma$ have the following hypergeometric expressions.
\begin{align*}
G_{6n}(\tau ) &= j(\tau )^{-n} {}_{2}F_{1} \left( \frac{1}{12}, \frac{5}{12} ; 1 ; \frac{1728}{j(\tau )} \right)^{6n-1} {}_{2}F_{1} \left( \frac{6n+1}{12}, \frac{6n+5}{12} ; n+1 ; \frac{1728}{j(\tau )} \right) , \\
G_{6n+2}(\tau ) &= j(\tau )^{-n}  {}_{2}F_{1} \left( \frac{1}{12}, \frac{5}{12} ; 1 ; \frac{1728}{j(\tau )}  \right)^{6n+1} {}_{2}F_{1} \left( \frac{6n-1}{12}, \frac{6n+7}{12} ; n+1 ; \frac{1728}{j(\tau )}  \right) , \\
G_{6n+4}(\tau ) &= E_{4}(\tau ) G_{6n}(\tau ) .
\end{align*}
\end{prop}

%%%%%%%%%%%%%%%%%%%%%%%%%%%%%%%%%%%%%%%%%%%%%%%%%%%%%%%%%%%%%%%%%%%%%%%%%%%%%%%%%%%%%%%%%
%%%%%%%%%%%%%%%%%%%%%%%%%%%%%%%%%%%%%%%%%%%%%%%%%%%%%%%%%%%%%%%%%%%%%%%%%%%%%%%%%%%%%%%%%
%%%%%%%%%%%%%%%%%%%%%%%%%%%%%%%%%%%%%%%%%%%%%%%%%%%%%%%%%%%%%%%%%%%%%%%%%%%%%%%%%%%%%%%%%

\section{Closed formulas of Atkin-like polynomials and their adjoint polynomials} \label{sec:clfAtkinpoly}

We first prove two recursion formulas that the normalized extremal quasimodular forms satisfy by computing certain differential operators. 
It is easy to rewrite the differential equation \eqref{eq:2ndKZeqn} to $L_{w}(f)=0$ by using \eqref{eq:exSerred}, where
\begin{align}
L_{w} = \partial_{w-1}^{2} -\frac{w^{2}-1}{144} E_{4}  . \label{eq:dopLw}
\end{align}
We also define the differential operators $K_{w}^{\mathrm{up}}$ as
\begin{align}
K_{w}^{\mathrm{up}} = E_{4} \partial_{w-1} - \frac{w+1}{12} E_{6} . \label{eq:dopKup} 
\end{align}
Then the normalized extremal quasimodular forms of depth 1 satisfy the following differential recursions (see \cite[Prop.~6.1]{grabner2020quasimodular}).
\begin{align}
&{} G_{0} = 1,\quad G_{w+6} = \frac{w+6}{72(w+1)(w+5)} K_{w}^{\mathrm{up}}(G_{w}), \label{eq:recGw0} \\
&{} G_{w+2} = \frac{12}{w+1} \partial_{w-1} (G_{w}) , \quad G_{w+4} = E_{4} G_{w}. \label{eq:recGw24}
\end{align}

\begin{prop} \label{prop:linrecG}
For $w\equiv 0 \pmod{6}$, the normalized extremal quasimodular forms $G_{w}$ and $G_{w+2}$ satisfy the following recursions:
\begin{align}
G_{w+12} &= - \frac{(w+6)(w+12)}{432(w+7)(w+11)} \left( E_{6} G_{w+6} - \Delta G_{w} \right) , \label{eq:d1linrecGw}  \\ 
G_{w+14} &= - \frac{(w+6)(w+12)}{432(w+5)(w+13)} \left( E_{6} G_{w+8} - \Delta G_{w+2} \right) . \label{eq:d1linrecGw+2}
\end{align}
For the initial values $G_{2}, G_{6}$, and $G_{8}$, see Section \ref{sec:intro}.
\end{prop}

\begin{proof}
We note that the following identity of the differential operators is valid.
\begin{align}
K_{w+6}^{\mathrm{up}} \circ K_{w}^{\mathrm{up}} = E_{4}^{2} L_{w} -\frac{w+6}{6} E_{6} K_{w}^{\mathrm{up}} + 12(w+1)(w+5) \Delta . \label{eq:Kupup}
\end{align}
Then we apply \eqref{eq:recGw0} inductively twice, and using $L_{w}(G_{w})=0$ and \eqref{eq:Kupup} we obtain the desired recursion \eqref{eq:d1linrecGw}:
\begin{align*}
G_{w+12} &= \frac{(w+6)(w+12)}{72^{2}(w+1)(w+5)(w+7)(w+11)} K_{w+6}^{\mathrm{up}} \circ K_{w}^{\mathrm{up}}(G_{w}) \\
&= - \frac{(w+6)(w+12)}{432(w+7)(w+11)} \left\{  \frac{w+6}{72(w+1)(w+5)} E_{6} K_{w}^{\mathrm{up}}(G_{w}) - \Delta G_{w} \right\}  \\
&= - \frac{(w+6)(w+12)}{432(w+7)(w+11)} \left(  E_{6} G_{w+6} - \Delta G_{w} \right) .
\end{align*}
Since we now assume $w \equiv 0 \pmod{6}$, we should not simply replace $w$ by $w+2$ in \eqref{eq:d1linrecGw}. Theredore, we focus on the following identity of the differential operators.
\begin{align}
\partial_{w+5} \circ K_{w}^{\mathrm{up}} = - \frac{6}{w-1} E_{4} L_{w} + \frac{w+5}{w-1} K_{w,2}^{\mathrm{up}} \circ \partial_{w-1}, \quad K_{w,2}^{\mathrm{up}} \coloneqq E_{4} \partial_{w+1} - \frac{w-1}{12} E_{6}. \label{eq:partialKup}
\end{align}
From this identity, we have the following formula for $G_{w+2}$, similar to \eqref{eq:recGw0}:
\begin{align*}
G_{w+8}^{(1)} &= \frac{12}{w+7} \partial_{w+5}(G_{w+6}) \quad (\text{by \eqref{eq:recGw24}}) \\
&= \frac{w+6}{6(w+1)(w+5)(w+7)} \partial_{w+5} \circ K_{w}^{\mathrm{up}}(G_{w}) \quad (\text{by \eqref{eq:recGw0}}) \\
&= \frac{w+6}{6(w-1)(w+1)(w+7)} K_{w,2}^{\mathrm{up}} \circ \partial_{w-1}(G_{w}) \quad (\text{by \eqref{eq:partialKup}}) \\
&= \frac{w+6}{72(w-1)(w+7)} K_{w,2}^{\mathrm{up}}(G_{w+2}) \quad (\text{by \eqref{eq:recGw24}}).
\end{align*}
The recursion formula \eqref{eq:d1linrecGw+2} is obtained by computing $K_{w+6,2}^{\mathrm{up}} \circ K_{w,2}^{\mathrm{up}}$ in the same way as for $G_{w}$, but then we use the following differential operator $L_{w,2}$ instead of $L_{w}$:
\begin{align*}
L_{w,2} \coloneqq E_{4} \partial_{w+1}^{2} +\frac{1}{3} E_{6} \partial_{w+1} - \frac{w^{2} -1}{144} E_{4}^{2}.
\end{align*}
Note that $L_{w,2}(G_{w+2}) =0$ because $L_{w,2} \circ \partial_{w-1} = \left( E_{4} \partial_{w+3} + \frac{1}{3} E_{6} \right) \circ L_{w}$. 
\end{proof}

\begin{rem}
\begin{enumerate}[(i)]
\item Recursions equivalent to \eqref{eq:d1linrecGw} and \eqref{eq:d1linrecGw+2} have already appeared in \cite[Prop.~1]{kaneko2003modular} and \cite[p.~270]{sakai2010modular}, respectively.
The difference in the appearance of these recursions in our paper and theirs is due to the different  normalization of the extremal quasimodular forms. 
\item We can take the recursions in Proposition \ref{prop:linrecG} to define  $G_{w}$ and $G_{w+2}$ instead of the differential recursions \eqref{eq:recGw0} and \eqref{eq:recGw24}. If we adopt this definition, it is clear from the recursions that the Fourier expansions of $G_{w}$ and $G_{w+2}$ are both of the form $q^{w/6}(1+O(q))$. 
Note also that if we rewrite these recursions using Proposition \ref{prop:Gto2F1}, they correspond to a contiguous relation formula of a certain hypergeometric series. 
\end{enumerate}
\end{rem}

Following the notation of \cite{kaneko2003modular}, we denote by $\eqref{eq:d1linrecGw}_{w+a}$ the expression in equation  \eqref{eq:d1linrecGw} with the parameter $w$ changed to $w+a$ (we assume that $a\equiv 0 \pmod{6}$).
Eliminating $G_{w+6}$ and $G_{w+18}$ from three equations $\eqref{eq:d1linrecGw}_{w}, \eqref{eq:d1linrecGw}_{w+6}$ and $\eqref{eq:d1linrecGw}_{w+12}$ yields
\begin{align}
\begin{split}
G_{w+24} &=  \frac{(w+12)(w+18)^{2}(w+24)}{2^{8} 3^{6} (w+13)(w+17)(w+19)(w+23)} \\
&{}\quad \times \left\{ \left( E_{4}^{3} - \frac{864(w^{2} +24w +103)}{(w+6)(w+18)} \Delta \right) G_{w+12} - \Delta^{2} G_{w} \right\} . \label{eq:d1linrecG24120}
\end{split}
\end{align}
In the same way for \eqref{eq:d1linrecGw+2}, we also obtain
\begin{align}
\begin{split}
G_{w+26} &= \frac{(w+12)(w+18)^{2}(w+24)}{2^{8} 3^{6} (w+11)(w+17)(w+19)(w+25)} \\
&{} \quad \times \left\{ \left( E_{4}^{3} - \frac{864(w^{2} +24w +115)}{(w+6)(w+18)} \Delta \right) G_{w+14} - \Delta^{2} G_{w+2} \right\} . \label{eq:d1linrecG26142}
\end{split}
\end{align}
Substituting expressions \eqref{eq:G12mAB} and \eqref{eq:G12m6AB} into \eqref{eq:d1linrecG24120}, and expressions \eqref{eq:G12m2AB} and \eqref{eq:G12m8AB} into \eqref{eq:d1linrecG26142}, we obtain the following three-term recursion formula, which is satisfied by the Atkin-like  polynomials and their adjoint polynomials: 
\begin{align}
W_{n+1,r}(X)=(X-a_{n,r})W_{n,r}(X)-b_{n,r}W_{n-1,r}(X) \quad (W \in \{A, B\}), \label{eq:d1linrecWAB}
\end{align}
where
\begin{align*}
a_{n,0} &= \frac{24 (144n^{2} -41)}{(2n+1)(2n-1)}, \quad b_{n,0} = \frac{36(12n-11)(12n-7)(12n-5)(12n-1)}{n(n-1)(2n-1)^{2}}  , \\
a_{n,2} &= \frac{24 (144n^{2} -29)}{(2n+1)(2n-1)}, \quad b_{n,2} = \frac{36(12n-13)(12n-7)(12n-5)(12n+1)}{n(n-1)(2n-1)^{2}} , \\
a_{n,6} &= a_{n+1/2,0}, \quad a_{n,8} = a_{n+1/2,2}, \quad b_{n,6} = b_{n+1/2,0}, \quad b_{n,8} = b_{n+1/2,2} . 
\end{align*}
Note that these recursions hold for $n\ge2$ if $r\in \{0,2\}$ and for $n\ge1$ if $r\in \{6,8\}$. The initial values are shown in the following table.
\begin{table}[H]
\caption{Initial values of $A_{n,r}(X)$ and $B_{n,r}(X)$.}
\begin{center}
\begin{tabular}{|c|c|c|c|c|} \hline
$n$ & $A_{n,0}(X)$ & $A_{n,2}(X)$ & $A_{n,6}(X)$ & $A_{n,8}(X)$ \\ \hline
$0$ & $0$ & $1$ & $1$ & $1$ \\ 
$1$ & $1$ & $X-720$ & $X-1266$ & $X-330$ \\ 
$2$ & $X-824$ & $X^{2}-1640X+269280$ &  &  \\ \hline 
$n$ & $B_{n,0}(X)$ & $B_{n,2}(X)$ & $B_{n,6}(X)$ & $B_{n,8}(X)$ \\ \hline
$0$ & $1$ & $0$ & $1$ & $1$ \\ 
$1$ & $X-1008$ & $1$ & $X-546$ & $X-1338$ \\ 
$2$ & $X^{2}-1832 X+497952$ & $X-920$ &  &  \\ \hline 
\end{tabular}
\end{center}
\end{table}

The following theorem is a generalization of Theorem 4 ii) in \cite{kaneko1998supersingular} to Atkin-like polynomials and their adjoint  polynomials.

\begin{thm} \label{thm:explicitformulaAB}
The Atkin-like polynomials $A_{n,r}(X)$ and their adjoint polynomials $B_{n,r}(X)$ have the following closed formula. These formulas are valid for $n\ge1$ if $r \in \{0,2\}$ and for $n\ge0$ if $r \in \{6,8\}$.
\begin{align*}
A_{n,r}(X) &= \sum_{i=0}^{d_{n,r}} 12^{3i} \phi_{i}(A_{n,r}) X^{d_{n,r} -i}, \quad B_{n,r}(X) = \sum_{i=0}^{d_{n,r+2}} 12^{3i} \phi_{i}(B_{n,r}) X^{d_{n,r+2} -i} , 
\end{align*}
where the degree is $d_{n,r}= \dim_{\mathbb{C}} M_{12n+r-2} -1 = \dim_{\mathbb{C}} S_{12n+r-2}$, and the coefficients are given by the following binomial sums:
\begin{align*}
\phi_{i}(A_{n,r}) &= \sum_{k=0}^{i} (-1)^{k} \binom{ -\tfrac{1}{12}}{i-k} \binom{ -\tfrac{5}{12}}{i-k} \binom{\kappa_{1}(A_{n,r})}{k} \binom{\kappa_{2}(A_{n,r})}{k} \binom{\kappa_{3}(A_{n,r})}{k}^{-1} , \\
\phi_{i}(B_{n,r}) &= \sum_{k=0}^{i} (-1)^{k} \binom{ \tfrac{1}{12}}{i-k} \binom{ -\tfrac{7}{12}}{i-k} \binom{\kappa_{1}(B_{n,r})}{k} \binom{\kappa_{2}(B_{n,r})}{k} \binom{\kappa_{3}(B_{n,r})}{k}^{-1} .
\end{align*}
Here the numbers $d_{n,r}, d_{n,r+2}$ and $\kappa_{i}(W_{n,r}) \; (1\le i \le 3,\, W\in \{A,B\})$ are summarized in the table below:
\begin{table}[H]
\caption{Parameters in the closed formulas of $A_{n,r}(X)$ and $B_{n,r}(X)$.}
\begin{center}
\begin{tabular}{|c|c|ccc|c|ccc|} \hline
\rule[-2pt]{0pt}{13pt}$r$ & $d_{n,r}$ & $\kappa_{1}(A_{n,r})$ & $\kappa_{2}(A_{n,r})$ & $\kappa_{3}(A_{n,r})$ & $d_{n,r+2}$ & $\kappa_{1}(B_{n,r})$ & $\kappa_{2}(B_{n,r})$ & $\kappa_{3}(B_{n,r})$ \\  \hline
\rule[-2pt]{0pt}{15pt}$0$ & $n-1$ & $n-\tfrac{11}{12}$ & $n-\tfrac{7}{12}$ & $2n-1$ & $n$ & $n-\tfrac{1}{12}$ & $n-\tfrac{5}{12}$ & $2n-1$ \\ 
\rule[-2pt]{0pt}{15pt}$2$ & $n$ & $n+\tfrac{1}{12}$ & $n-\tfrac{7}{12}$ & $2n-1$ & $n-1$ & $n-\tfrac{13}{12}$ & $n-\tfrac{5}{12}$ & $2n-1$ \\ 
\rule[-2pt]{0pt}{15pt}$6$ & $n$ & $n+\tfrac{1}{12}$ & $n+\tfrac{5}{12}$ & $2n$ & $n$ & $n-\tfrac{1}{12}$ & $n-\tfrac{5}{12}$ & $2n$ \\ 
\rule[-6pt]{0pt}{19pt}$8$ & $n$ & $n+\tfrac{1}{12}$ & $n-\tfrac{7}{12}$ & $2n$ & $n$ & $n-\tfrac{1}{12}$ & $n+\tfrac{7}{12}$ & $2n$ \\  \hline
\end{tabular}
\end{center}
\end{table}
\end{thm}

\begin{proof}
We define three monic polynomials $\alpha_{n}^{\varepsilon}(X)\, (\varepsilon \in \{0,1\}), \beta_{n}(X)$ of degree $n\ge0$ by
\begin{align*}
X^{n} {}_{2}F_{1} \left( \tfrac{1}{12}, \tfrac{5}{12} ; 1 ; \tfrac{1728}{X}  \right) &= \alpha_{n}^{0}(X) + O(1/X), \\
X^{n-1} (X-1728)\, {}_{2}F_{1} \left( \tfrac{7}{12}, \tfrac{11}{12} ; 1 ; \tfrac{1728}{X}  \right) &= \alpha_{n}^{1}(X) + O(1/X), \\
X^{n} {}_{2}F_{1} \left( -\tfrac{1}{12}, \tfrac{7}{12} ; 1 ; \tfrac{1728}{X}  \right) &= \beta_{n}(X) + O(1/X).
\end{align*}
(The polynomial $\alpha_{n}^{1}(X)$ will be used later.) Then the assertion of Theorem \ref{thm:explicitformulaAB} is equivalent to the following expressions.
\begin{align}
A_{n,r}(X) &= \sum_{k=0}^{d_{n,r}} (-12)^{3k} \binom{\kappa_{1}(A_{n,r})}{k} \binom{\kappa_{2}(A_{n,r})}{k} \binom{\kappa_{3}(A_{n,r})}{k}^{-1} \alpha_{d_{n,r} -k}^{0}(X), \label{eq:Atkinlikepolybyalpha} \\
B_{n,r}(X) &= \sum_{k=0}^{d_{n,r+2}} (-12)^{3k} \binom{\kappa_{1}(B_{n,r})}{k} \binom{\kappa_{2}(B_{n,r})}{k} \binom{\kappa_{3}(B_{n,r})}{k}^{-1} \beta_{d_{n,r+2} -k}(X) . \label{eq:Bpolybybeta}
\end{align}
We can check directly that these equations satisfy the recursion \eqref{eq:d1linrecWAB}, and the calculation is similar to the proof of Proposition 4 in \cite{kaneko1998supersingular}, so we omit it. 
\end{proof}

\begin{rem}
As already mentioned in \cite[p.~116]{kaneko1998supersingular}, the Atkin polynomial $A_{n,2}(X)$ is obtained as a truncation of the product of two hypergeometric series ${}_{2}F_{1}$, and this fact also holds for $A_{n,r}(X)$ and $B_{n,r}(X)$. Their hypergeometric representation is simply a rewriting of equations \eqref{eq:Atkinlikepolybyalpha} and \eqref{eq:Bpolybybeta}. Specifically, $A_{n,r}(X)$ and $B_{n,r}(X)$ are the polynomial parts of the following two equations, respectively:
\begin{align*}
&{} X^{d_{n,r}} {}_{2}F_{1} \left( \frac{1}{12}, \frac{5}{12} ; 1 ; \frac{1728}{X} \right) {}_{2}F_{1} \left( -\kappa_{1}(A_{n,r}), -\kappa_{2}(A_{n,r}) ; -\kappa_{3}(A_{n,r}) ; \frac{1728}{X} \right), \\
&{} X^{d_{n,r+2}} {}_{2}F_{1} \left( -\frac{1}{12}, \frac{7}{12} ; 1 ; \frac{1728}{X} \right) {}_{2}F_{1} \left( -\kappa_{1}(B_{n,r}), -\kappa_{2}(B_{n,r}) ; -\kappa_{3}(B_{n,r}) ; \frac{1728}{X} \right).
\end{align*}
Thus, using the results of \cite[p.~244]{bailey1928products} on products of generalized hypergeometric series, we can express the coefficients $\phi_{i}(A_{n,r})$ and $\phi_{i}(B_{n,r})$ as a special value of a generalized hypergeometric series ${}_{4}F_{3}$. 
For example, we have 
\begin{align*}
\phi_{i}(A_{n,2}) &= \frac{(\tfrac{1}{12})_{i} (\tfrac{5}{12})_{i}}{(i!)^{2}} \, {}_{4}F_{3} \left(-i, -i, -n-\tfrac{1}{12}, -n+\tfrac{7}{12} ; -i+\tfrac{7}{12}, -i+\tfrac{11}{12}, -2n+1 ; 1 \right) , \\
\phi_{i}(B_{n,2}) &= \frac{(-\tfrac{1}{12})_{i} (\tfrac{7}{12})_{i}}{(i!)^{2}} \, {}_{4}F_{3} \left(-i, -i, -n+\tfrac{5}{12}, -n+\tfrac{13}{12} ; -i+\tfrac{5}{12}, -i+\tfrac{13}{12}, -2n+1 ; 1 \right) .
\end{align*}
\end{rem}

To prove a congruence formula for the Atkin-like polynomials $A_{n,r}(X)$, we also need the following formulas.  They can be proved in the same way as Theorem~\ref{thm:explicitformulaAB}.

\begin{prop} \label{prop:Atkinalpha}
The Atkin-like polynomials $A_{n,r}(X)$, multiplied by the appropriate factors, have the  following expansions with respect to the polynomials $\alpha_{n}^{\varepsilon}(X)$. Unless otherwise stated, these expressions are valid for $n \ge 1$.
\begin{align*}
X A_{n,0}(X) &= \sum_{k=0}^{n} (-12)^{3k} \binom{n-\tfrac{11}{12}}{k} \binom{n-\tfrac{7}{12}}{k} \binom{2n-1}{k}^{-1} \alpha_{n-k}^{0}(X)  , \\
(X-1728) A_{n,0}(X) &= \sum_{k=0}^{n} (-12)^{3k} \binom{n-\tfrac{1}{12}}{k} \binom{n-\tfrac{5}{12}}{k} \binom{2n-1}{k}^{-1} \alpha_{n-k}^{1}(X)  , \\
X(X-1728) A_{n,0}(X) &= \sum_{k=0}^{n+1} (-12)^{3k} \binom{n-\tfrac{1}{12}}{k} \binom{n-\tfrac{5}{12}}{k} \binom{2n-1}{k}^{-1} \alpha_{n-k+1}^{1}(X) \quad (n\ge2),  \\
A_{n,2}(X) &= \sum_{k=0}^{n} (-12)^{3k} \binom{n-\tfrac{13}{12}}{k} \binom{n-\tfrac{5}{12}}{k} \binom{2n-1}{k}^{-1} \alpha_{n-k}^{1}(X) , \\
X A_{n,6}(X) &= \sum_{k=0}^{n+1} (-12)^{3k} \binom{n-\tfrac{1}{12}}{k} \binom{n-\tfrac{5}{12}}{k} \binom{2n}{k}^{-1} \alpha_{n-k+1}^{1}(X) , \\
(X-1728) A_{n,8}(X) &= \sum_{k=0}^{n+1} (-12)^{3k} \binom{n-\tfrac{1}{12}}{k} \binom{n+\tfrac{7}{12}}{k} \binom{2n}{k}^{-1} \alpha_{n-k+1}^{1}(X) .
\end{align*}
\end{prop}

\begin{thm} \label{thm:congruenceAtkin-like}
Let $p \ge 5$ be a prime number and put $p-1=12m+4\delta +6 \varepsilon$ with $m \in \mathbb{Z}_{\ge0} , \delta, \varepsilon \in \{0,1\}$. Then we have
\begin{align*}
A_{m +\delta +\varepsilon,2}(X) &\equiv X^{\delta} A_{m+\varepsilon,6}(X) \equiv (X -1728)^{\varepsilon} A_{m +\delta,8}(X) \\
&\equiv X^{\delta} (X -1728)^{\varepsilon} A_{m+1,0}(X) \pmod{p}. 
\end{align*}
\end{thm}
\begin{proof}
Since the polynomial $A_{n,2}(X)$ is nothing else than the Atkin polynomial $A_{n}(X)$, we see that $A_{m+\delta+\varepsilon, 2}(X) \equiv \alpha_{m+\delta+\varepsilon}^{\varepsilon}(X) \pmod{p}$ from Propositions~4 and~5 of \cite{kaneko1998supersingular}.
By equation \eqref{eq:Atkinlikepolybyalpha} and Proposition \ref{prop:Atkinalpha}, which expresses the polynomial $A_{n,r}(X)$ (with some factors) in terms of the polynomial $\alpha_{*}^{\varepsilon}(X)$, we see that the remaining polynomials are also congruent to $\alpha_{m+\delta+\varepsilon}^{\varepsilon}(X)$ modulo~$p$, since all binomial coefficients except the one for $k=0$ vanish modulo~$p$.
\end{proof}

\begin{rem}
\begin{enumerate}[(i)]
\item Our Theorem \ref{thm:congruenceAtkin-like} corresponds to a specialization of Theorem 1 in \cite{basha2004systems}, but they and we use different proof methods. 
\item As mentioned in the introduction, the mod $p$ reduction of the Atkin  polynomial is equal to the supersingular polynomial $ss_{p}(X)$ whose roots are exactly the $j$-invariants of all the supersingular elliptic curves in characteristic $p$. 
Surprisingly, it is known that the supersingular polynomial $ss_{p}(X)$, and thus the Atkin polynomial $A_{m +\delta +\varepsilon,2}(X) \pmod{p}$, factorizes only into linear factors if and only if $p \mid \# \mathbb{M}$, where $m ,\delta ,\varepsilon$ have the same meaning as in Theorem~\ref{thm:congruenceAtkin-like} (for $p=2,3$ we have $ss_{p}(X) \equiv A_{1,2}(X) \equiv X \pmod{p}$) and $\mathbb{M}$ is the Monster group, which is the largest sporadic finite simple group of order 
\begin{align*}
\# \mathbb{M} &=808017424794512875886459904961710757005754368000000000 \\
&=2^{46}\cdot 3^{20}\cdot 5^9\cdot 7^6\cdot 11^2\cdot 13^3\cdot 17\cdot 19\cdot 23\cdot 29\cdot 31\cdot 41\cdot 47\cdot 59\cdot 71.
\end{align*}
This coincidence is known as the ``Jack Daniels problem". See \cite{ogg1975automorphismes} and \cite{duncan2016jack} for details.

Interestingly, this coincidence also occurs in the supersingular polynomials for the Fricke group of levels 2,3,5 and 7, and the sporadic finite simple groups 
Baby Monster group $\mathbb{B}$, largest Fisher group $Fi'_{24}$, Harada--Norton group $HN$, and Held group $He$, respectively. 
See \cite[Thm.~5]{nakaya2019supersingular}, \cite{sakai2011atkin,sakai2014atkin}, and \cite[Thm.~6]{morton2023supersingular} for more details. Note that of these, the relationship between the supersingular polynomials and the Atkin polynomials has only been proved for levels 2 and~3.
\end{enumerate}
\end{rem}

%%%%%%%%%%%%%%%%%%%%%%%%%%%%%%%%%%%%%%%%%%%%%%%%%%%%%%%%%%%%%%%%%%%%%%%%%%%%%%%%%%%%%%%%%
%%%%%%%%%%%%%%%%%%%%%%%%%%%%%%%%%%%%%%%%%%%%%%%%%%%%%%%%%%%%%%%%%%%%%%%%%%%%%%%%%%%%%%%%%
%%%%%%%%%%%%%%%%%%%%%%%%%%%%%%%%%%%%%%%%%%%%%%%%%%%%%%%%%%%%%%%%%%%%%%%%%%%%%%%%%%%%%%%%%

\section{Orthogonality of the Atkin-like polynomials} \label{sec:orthogonalityA}

Let $\mathcal{L}$ be the linear functional corresponding to the Atkin inner product \eqref{def:Atkininnerprod}. Then the moments $\mathcal{L}(j^{n})\;(n \ge 0)$ can be expressed as
\begin{align}
\mathcal{L}(j^{n}) = (j^{n},1) = \mathop{\mathrm{Res}}_{q=0} j(\tau )^{n} E_{2}(\tau ) \frac{d q}{q}. \label{eq:momentL}
\end{align}
Note that since the Fourier coefficients of $j(\tau )$ and $E_{2}(\tau )$ are integers, the value of $\mathcal{L}(j^{n})$ is also an integer.

Let $p=e^{2\pi i \sigma}, \; \sigma \in \mathfrak{H}$ and assume that $\Im(\sigma)$ is sufficiently large. Following the standard abuse of notation, we now assume that $f(\sigma)$ and $f(p)$ represent the same function $f$. 
By formally expanding $1/(j(p)-j)$ to the power series $\sum_{n=0}^{\infty} j^{n}j(p)^{-n-1}$, 
the moment-generating function\footnote{In the context of orthogonal polynomial theory, it is called the Stieltjes function (see, e.g., \cite{zhedanov1997rational}).} associated with $\mathcal{L}$ can be rewritten as follows:
\begin{align}
\mathcal{L} \left( \frac{1}{j(p) - j} \right) = \sum_{n=0}^{\infty} \mathcal{L}(j^{n}) j(p)^{-n-1} = \frac{E_{2}(p)E_{4}(p)}{j(p)E_{6}(p)} = \frac{{}_{2}F_{1} \left( \frac{5}{12}, \frac{13}{12} ; 1 ; \frac{1728}{j(p)} \right) }{j(p)\, {}_{2}F_{1} \left( \frac{1}{12}, \frac{5}{12} ; 1 ; \frac{1728}{j(p)} \right) }.  \label{eq:Stieltjes function}
\end{align}
Here and below we assume that the linear functional $\mathcal{L}$ acts only on the variable $j$. The second equality in \eqref{eq:Stieltjes function} is found in \cite[\S 5]{kaneko1998supersingular}, and the last equality is obtained from Proposition \ref{prop:EisensteinHyp} and \eqref{eq:Euler}.

The Atkin polynomial $A_{n,2}(X)$ is the orthogonal polynomial corresponding to $\mathcal{L}$, and the orthogonality relation $\mathcal{L} \left(A_{m,2}(j)A_{n,2}(j) \right) = \mathcal{L} \left(A_{m,2}(j)^{2} \right) \delta_{m,n}$ holds for $m,n\ge0$, where the symbol $\delta_{m,n}$ denotes the Kronecker delta, and especially $\mathcal{L} \left(A_{n,2}(j) \right) =0$ holds for $n\ge1$. 
We first prove analogues of the latter formula for the Atkin-like polynomials $A_{n,r}(j)$ for $r \in \{0,6,8\}$. 
The explicit orthogonality relations that correspond to the former formula will be given later in Theorem~\ref{prop:imageofAtkinL}.

\begin{prop} \label{prop:orthogonality}
The following equation holds for any integer $n\ge1$.
\begin{align*}
\mathcal{L} \left( j A_{n,6}(j) \right) = \mathcal{L} \left( (j-1728) A_{n,8}(j) \right) =  \mathcal{L} \left( j(j-1728) A_{n+1,0}(j) \right) =0.
\end{align*}
\end{prop}
\begin{proof}
By comparing the coefficients of $\alpha_{\ell}^{1}(j)$ in Proposition \ref{prop:Atkinalpha}, we obtain the following expansion formulas for $n\ge1$:
\begin{align}
&{} j A_{n,6}(j) = A_{n+1,2}(j) + \frac{6(12n-1)(12n+5)}{n(2n+1)} A_{n,2}(j) , \label{eq:recAn6An2} \\ 
&{} (j-1728) A_{n,8}(j) = A_{n+1,2}(j) - \frac{6(12n-1)(12n+7)}{n(2n+1)} A_{n,2}(j) , \label{eq:recAn8An2} \\
&{} j (j-1728) A_{n+1,0}(j) = j A_{n+1,6}(j) - \frac{6(12n+7)(12n+11)}{(n+1)(2n+1)} j A_{n,6}(j) \label{eq:recAn0An6} \\
 &= (j-1728) A_{n+1,8}(j) + \frac{6(12n+5)(12n+11)}{(n+1)(2n+1)} (j-1728) A_{n,8}(j)  \label{eq:recAn0An8} \\
\begin{split}
&= A_{n+2,2}(j) - \frac{24(12n+11)}{(2n+1)(2n+3)} A_{n+1,2}(j) \\
&{} \quad - \frac{36(12n-1)(12n+5)(12n+7)(12n+11)}{n(n+1)(2n+1)^{2}} A_{n,2}(j) .
\end{split} \label{eq:recAn0An2}
\end{align}
(These equations can also be proved by the fact that each side satisfies the same recursion.) By letting $\mathcal{L}$ act on both sides of \eqref{eq:recAn6An2}, \eqref{eq:recAn8An2}, and \eqref{eq:recAn0An2}, the orthogonality $\mathcal{L} \left(A_{n,2}(j) \right) =0 \, (n\ge1)$ gives us the desired results.
\end{proof}

%From the viewpoint of orthogonal polynomial theory, the above equations \eqref{eq:recAn6An2} to \eqref{eq:recAn0An8} have clear meaning. 
The above equations \eqref{eq:recAn6An2} to \eqref{eq:recAn0An8} have a clear meaning from the point of view of orthogonal polynomial theory.
For a monic orthogonal polynomial $p_{n}(x)$, its Christoffel transform (CT, for brevity) $q_{n}(x)$ is defined by the formula 
\begin{align}
q_{n}(x) = q_{n}(x;\lambda) = \frac{1}{x-\lambda} \left( p_{n+1}(x) - \frac{p_{n+1}(\lambda)}{p_{n}(\lambda)} p_{n}(x) \right), \label{eq:Christoffel transform}
\end{align}
where $\lambda$ is an arbitrary parameter with $p_{n}(\lambda)\not=0$. 
Let $L$ be the linear functional corresponding to the orthogonal polynomials $p_{n}(x)$. If the new linear functional $L^{*}$ is defined by $L^{*}(f(x))=L((x-\lambda)f(x))$, then the above polynomial $q_{n}(x)$ is the monic orthogonal polynomial corresponding to $L^{*}$ (see \cite[Thm.~7.1]{chihara1978intortho}. This assertion is a specialization of \cite[Thm.~2.5]{szego1975orthogonal}). 

Using the recursion formula \eqref{eq:d1linrecWAB}, we obtain the following special values of the Atkin-like polynomial for $n\ge1$ (Note that the first two equations have already been proved in \cite[Prop.~6]{kaneko1998supersingular}):
\begin{align*}
&{} A_{n,2}(0) = (-12)^{3n+1} \frac{(-\tfrac{1}{12})_{n}(\tfrac{5}{12})_{n}}{(2n-1)!}, \quad A_{n,2}(1728)= - 12^{3n+1} \frac{(-\tfrac{1}{12})_{n}(\tfrac{7}{12})_{n}}{(2n-1)!} , \\
&{} A_{n,6}(1728) = 12^{3n} \frac{(\tfrac{7}{12})_{n}(\tfrac{11}{12})_{n}}{(2n)!} , \quad A_{n,8}(0) = (-12)^{3n} \frac{(\tfrac{5}{12})_{n}(\tfrac{11}{12})_{n}}{(2n)!},
\end{align*}
and so
\begin{align*}
&{} \frac{A_{n+1,2}(0)}{A_{n,2}(0)} = - \frac{6(12n-1)(12n+5)}{n(2n+1)}, \quad \frac{A_{n+1,2}(1728)}{A_{n,2}(1728)} = \frac{6(12n-1)(12n+7)}{n(2n+1)}, \\
&{} \frac{A_{n+1,6}(1728)}{A_{n,6}(1728)} = \frac{6(12n+7)(12n+11)}{(n+1)(2n+1)}, \quad \frac{A_{n+1,8}(0)}{A_{n,8}(0)} = - \frac{6(12n+5)(12n+11)}{(n+1)(2n+1)}.
\end{align*}
Therefore, from equations \eqref{eq:recAn6An2} and \eqref{eq:recAn8An2}, the polynomials $A_{n,6}(X)$ and $A_{n,8}(X)$ are CT of $A_{n,2}(X)$ corresponding to $\lambda =0,1728$ respectively. Similarly,  from equations \eqref{eq:recAn0An6} and \eqref{eq:recAn0An8}, the polynomial $A_{n+1,0}(X)$ is CT of $A_{n,6}(X)$ and $A_{n,8}(X)$ corresponding to $\lambda =1728,0$ respectively. 
Note that there is the inverse transform of CT called the Geronimus transform (\cite[p.~71]{zhedanov1997rational}), e.g., we have
\begin{align*}
A_{n,2}(X) = A_{n,6}(X) + \frac{6(12n+1)(12n-5)}{n(2n-1)} A_{n-1,6}(X) \quad (n\ge1).
\end{align*}

\begin{prop} \label{prop:orthpolyexp}
The following expansions hold for any integer $n\ge0$. 
\begin{align}
A_{n,6}(X) &= \sum_{k=0}^{n} (-12)^{3k} \binom{n+\tfrac{1}{12}}{k} \binom{n-\tfrac{5}{12}}{k} \binom{2k}{k}^{-1} \binom{2n}{2k}^{-1} A_{n-k,2}(X) , \label{eq:expAn6byAn2} \\
A_{n,8}(X) &= \sum_{k=0}^{n} 12^{3k} \binom{n+\tfrac{1}{12}}{k} \binom{n-\tfrac{7}{12}}{k} \binom{2k}{k}^{-1} \binom{2n}{2k}^{-1} A_{n-k,2}(X), \\
A_{n+1,0}(X) &= \sum_{k=0}^{n} 12^{3k} \binom{n+\tfrac{1}{12}}{k} \binom{n+\tfrac{5}{12}}{k} \binom{2k}{k}^{-1} \binom{2n+1}{2k}^{-1} A_{n-k,6}(X) \label{eq:expAn0byAn6} \\
&= \sum_{k=0}^{n} (-12)^{3k} \binom{n+\tfrac{1}{12}}{k} \binom{n+\tfrac{7}{12}}{k} \binom{2k}{k}^{-1} \binom{2n+1}{2k}^{-1} A_{n-k,8}(X) \\
&= \sum_{k=0}^{n} (-12)^{3k} C_{n,k} \binom{n+\tfrac{1}{12}}{k} \binom{2k+1}{k}^{-1} \binom{2n+1}{2k+1}^{-1} A_{n-k,2}(X),
\end{align}
where $C_{n,k}=(-1)^{k} \dbinom{n+\tfrac{5}{12}}{k+1} + \dbinom{n+\tfrac{7}{12}}{k+1}$.
\end{prop}
To prove the last equality of the above proposition, we prepare the following lemma. 
\begin{lem}\label{lem:Cnr}
For arbitrary integers $n,r \ge 0$, the following equation holds.
\begin{align}
(r+1)! \, C_{n,r} = \sum_{k=0}^{r} (-1)^{k} (2n-2k+1) \binom{n+\tfrac{5}{12}}{k} \binom{n-k-\tfrac{5}{12}}{r-k} k! (r-k)! . \label{eq:Cnr}
\end{align}
\end{lem}
\begin{proof}
Let the right-hand side of equation \eqref{eq:Cnr} be $D_{n,r}$. Then, for a fixed $n$, it is easy to see that $D_{n,r}$ satisfies the following recursion: 
\begin{align*}
D_{n,r+1} = (-1)^{r+1} (2n-2r-1) \binom{n+\tfrac{5}{12}}{r+1} (r+1)!  + \left( n-r-\frac{5}{12} \right) D_{n,r}.
\end{align*}
Since $(r+1)! \, C_{n,r}$ also satisfies the same recursion and the initial value $1! \, C_{n,0}=F_{n,0}=2n+1$ coincide, the claim follows.
\end{proof}

\begin{proof}[Proof of Proposition \ref{prop:orthpolyexp}]
First, we focus on the Christoffel--Darboux identity \cite[Thm.~4.5]{chihara1978intortho} for the Atkin polynomials $A_{n,2}(X)$, which can be proved repeatedly using the recursion formula \eqref{eq:d1linrecWAB}:
\begin{align*}
&{} \frac{A_{n+1,2}(X)A_{n,2}(Y) - A_{n,2}(X) A_{n+1,2}(Y)}{X-Y} \\
&= \sum_{k=0}^{n} \left( \prod_{\ell=k+1}^{n} b_{\ell ,2} \right) A_{k,2}(X) A_{k,2}(Y) \quad (n\ge0),
\end{align*}
where the empty product is be understood as 1. 
Dividing both sides by $A_{n,2}(Y)$, the left-hand side gives the Christoffel transform \eqref{eq:Christoffel transform}, so the expansion of $A_{n,6}(X)$ (resp., $A_{n,8}(X)$) by $A_{n,2}(X)$ is obtained by setting $Y = 0$ (resp., $Y=1728$) and transforming the coefficients appropriately. 
Similarly, by applying the Christoffel--Darboux identity to $A_{n,6}(X)$ and $A_{n,8}(X)$, we obtain the expansion formula for $A_{n+1,0}(X)$. The last equality is obtained by substituting \eqref{eq:expAn6byAn2} into \eqref{eq:expAn0byAn6} and transforming it as follows. By setting $k+l=r$  we have
\begin{align*}
&{}A_{n+1,0}(X) = \sum_{k=0}^{n} 12^{3k} \cdots \sum_{l=0}^{n-k} (-12)^{3l} \cdots A_{n-k-l,2}(X) \\
&= \sum_{r=0}^{n} (-12)^{3r} \sum_{k=0}^{r} (-1)^{k} \cdots A_{n-r,2}(X) \\
&= \sum_{r=0}^{n} (-12)^{3r} \frac{(2n-2r)! r!}{(2n+1)!} \binom{n+\tfrac{1}{12}}{r} (\text{the right-hand side of \eqref{eq:Cnr}}) A_{n-r,2}(X).
\end{align*}
(The part of the above equation omitted by the dotted line is the product of the appropriate binomial coefficients.) Thus, with the help of Lemma \ref{lem:Cnr}, the proof is complete. 
\end{proof}

\begin{thm}\label{prop:imageofAtkinL} 
\begin{enumerate}[(i)]
\item For any integer $m\ge0$, we have
\begin{align*}
B_{m+1,0}(X) &= \mathcal{L} \left(  \frac{ X(X-1728) A_{m+1,0}(X) - j(j-1728) A_{m+1,0}(j)}{X-j} \right) , \\
B_{m,2}(X) &= \mathcal{L} \left(  \frac{A_{m,2}(X) - A_{m,2}(j)}{X-j} \right) , \\
B_{m,6}(X) &= \mathcal{L} \left(  \frac{X A_{m,6}(X) - j A_{m,6}(j)}{X-j} \right) , \\
B_{m,8}(X) &= \mathcal{L} \left(  \frac{(X-1728) A_{m,8}(X) - (j-1728) A_{m,8}(j)}{X-j} \right) .
\end{align*}
\item  For any integer $m\ge0$, the following equations hold as a generalization of \eqref{eq:Stieltjes function}. Here, the normalizing factor $N_{m,a}$ is given by \eqref{eq:d1nf02} and \eqref{eq:d1nf68}.
\begin{align*}
&{} \mathcal{L} \left(  \frac{j(j-1728) A_{m+1,0}(j)}{j(p)-j} \right) = - N_{m+1,0}\, \frac{{}_{2}F_{1} \left( m+\frac{13}{12}, m+\frac{17}{12} ; 2m+3 ; \frac{1728}{j(p)} \right) }{j(p)^{m+1}\, {}_{2}F_{1} \left( \frac{1}{12}, \frac{5}{12} ; 1 ; \frac{1728}{j(p)} \right) } , \\
&{} \mathcal{L} \left(  \frac{A_{m,2}(j)}{j(p)-j} \right)  = N_{m,2}\, \frac{{}_{2}F_{1} \left( m+\frac{5}{12}, m+\frac{13}{12} ; 2m+1 ; \frac{1728}{j(p)} \right) }{j(p)^{m+1}\, {}_{2}F_{1} \left( \frac{1}{12}, \frac{5}{12} ; 1 ; \frac{1728}{j(p)} \right) } , \\
&{} \mathcal{L} \left(  \frac{j A_{m,6}(j)}{j(p)-j} \right) = N_{m,6}\, \frac{{}_{2}F_{1} \left( m+\frac{13}{12}, m+\frac{17}{12} ; 2m+2 ; \frac{1728}{j(p)} \right) }{j(p)^{m+1}\, {}_{2}F_{1} \left( \frac{1}{12}, \frac{5}{12} ; 1 ; \frac{1728}{j(p)} \right) } , \\
&{} \mathcal{L} \left(  \frac{(j-1728) A_{m,8}(j)}{j(p)-j} \right) = - N_{m,8}\, \frac{{}_{2}F_{1} \left( m+\frac{5}{12}, m+\frac{13}{12} ; 2m+2 ; \frac{1728}{j(p)} \right) }{j(p)^{m+1}\, {}_{2}F_{1} \left( \frac{1}{12}, \frac{5}{12} ; 1 ; \frac{1728}{j(p)} \right) } .
\end{align*}
\item For any integers $m,n\ge0$, we have the following orthogonality relations:
\begin{align*}
&{} \mathcal{L} \left( j(j-1728) A_{m+1,0}(j) A_{n+1,0}(j) \right) = - N_{m+1,0} \, \delta_{m,n} , \\
&{} \mathcal{L} \left( A_{m,2}(j) A_{n,2}(j) \right) = N_{m,2}\, \delta_{m,n},  \\
&{} \mathcal{L} \left( j A_{m,6}(j) A_{n,6}(j) \right) = N_{m,6}\, \delta_{m,n},  \\
&{} \mathcal{L} \left( (j-1728) A_{m,8}(j) A_{n,8}(j) \right) = - N_{m,8}\, \delta_{m,n} .
\end{align*}
\end{enumerate}
\end{thm}

\begin{proof}
In assertions (i) to (iii), we only prove the cases related to $A_{m,6}(X)$ or $B_{m,6}(X)$, and the other cases can be proved in a similar way. 
\begin{enumerate}[(i)]
\item Denote the expression on the right-hand side of the corresponding formula by $Z_{m}(X)$.
We can check directly that $Z_{m}(X)=B_{m,6}(X)$ for $m=0$ and $1$. For $m\ge1$, since $A_{m,6}(X)$ satisfies the recursion \eqref{eq:d1linrecWAB}, we have
\begin{align*}
&{} (X-a_{m,6})(X A_{m,6}(X) - j A_{m,6}(j)) - b_{m,6} (X A_{m-1,6}(X) - j A_{m-1,6}(j)) \\
&= X \{(X-a_{m,6})A_{m,6}(X) -b_{m,6}A_{m-1,6}(X)\} \\
&{} \quad - j \{(j-a_{m,6})A_{m,6}(j) -b_{m,6}A_{m-1,6}(j) \} - j (X-j) A_{m,6}(j) \\
&= X A_{m+1,6}(X) - j A_{m+1,6}(j) - j (X-j) A_{m,6}(j).
\end{align*}
Dividing both sides of the above equation by $X-j$ and acting $\mathcal{L}$, we have
\begin{align*}
&{} (X-a_{m,6}) Z_{m}(X) - b_{m,6} Z_{m-1}(X) = Z_{m+1}(X) - \mathcal{L} \left( j A_{m,6}(j) \right). 
\end{align*}
We see that the second term of the right-hand side vanishes from Proposition~\ref{prop:orthogonality}, and therefore the polynomials $Z_{m}(X)$ and $B_{m,6}(X)$ have the same initial values and satisfy the same recursion, we conclude that $Z_{m}(X) = B_{m,6}(X)$ for $m\ge0$. 

\item One can easily see that
\begin{align*}
\mathcal{L} \left(  \frac{j A_{m,6}(j)}{j(p)-j} \right) &= \mathcal{L} \left( \frac{j(p)A_{m,6}(j(p))}{j(p) - j} \right) - \mathcal{L} \left(  \frac{j(p)A_{m,6}(j(p)) - j A_{m,6}(j)}{j(p)-j} \right) \\
&= j(p)A_{m,6}(j(p)) \mathcal{L} \left( \frac{1}{j(p) - j} \right) - B_{m,6}(j(p)) \quad \text{(by (i))}\\
&= \frac{E_{2}(p)E_{4}(p) A_{m,6}(j(p)) -  E_{6}(p) B_{m,6}(j(p)) }{E_{6}(p)} \quad  \text{(by \eqref{eq:Stieltjes function})} \\
&= \frac{N_{m,6}\, G_{12m+6}(p)}{E_{6}(p) \Delta(p)^{m}} \quad \text{(by \eqref{eq:G12m6AB})}.
\end{align*}
The proof is completed by transforming this equation into a hypergeometric series using Propositions \ref{prop:EisensteinHyp} and \ref{prop:Gto2F1} and the transformation formula \eqref{eq:Euler}. 

As an alternative proof, we first show the result for $\mathcal{L} \left( A_{m,2}(j)/(j(p)-j) \right)$, then transform $\mathcal{L} \left( j A_{m,6}(j)/(j(p)-j) \right)$ using the expansion \eqref{eq:recAn6An2}, 
and obtain the desired form by using the appropriate contiguous relation of a hypergeometric series. 

\item By comparing the coefficients of $j(p)^{-\ell}$ in 
\begin{align*}
\sum_{k=0}^{\infty} \mathcal{L}\left( j^{k+1} A_{m,6}(j) \right)\, j(p)^{-k-1} = \mathcal{L} \left(  \frac{j A_{m,6}(j)}{j(p)-j} \right) = \frac{N_{m,6}}{j(p)^{m+1}} \left( 1 + O \left(\frac{1}{j(p)}\right) \right)
\end{align*}
obtained from assertion (ii), we have
\begin{align*}
\mathcal{L}\left( j^{k+1} A_{m,6}(j) \right)=0 \;\; (0\le k \le m-1) , \quad \mathcal{L}\left( j^{m+1} A_{m,6}(j) \right) = N_{m,6}.
\end{align*}
Furthermore, by using this orthogonality relation, we also have 
\begin{align*}
\mathcal{L}\left( j A_{m,6}(j)^{2} \right) = \mathcal{L}\left( j (j^{m} + (\text{lower order terms})) A_{m,6}(j) \right) =\mathcal{L}\left( j^{m+1} A_{m,6}(j) \right) = N_{m,6}.
\end{align*} 
Summarizing these results, we finally conclude that $\mathcal{L} \left( j A_{m,6}(j) A_{n,6}(j) \right) = N_{m,6}\, \delta_{m,n}$.
\end{enumerate}
This completes the proof. 
\end{proof}

\begin{rem}\label{rem:Atkinweightfunction}
\begin{enumerate}[(i)]
\item The Atkin inner product \eqref{def:Atkininnerprod} has the following integral representation (see \cite[\S 5]{kaneko1998supersingular}). For $f,g \in \mathbb{C}[j]$, we have
\begin{align*}
(f,g)= \int_{0}^{1728} f(j)g(j) w(j) d j, \quad w(j) = \frac{6}{\pi} \theta' (j).
\end{align*}
Here the function $\theta : [0,1728] \rightarrow [\pi/3, \pi/2]$ is the inverse of the monotone increasing function $\theta \mapsto j(e^{i\theta})$ and then $w(j)>0$. 
For a concrete expression of the weight function $w(j)$ by a hypergeometric series, the interested reader is referred to \cite[Thm.~7.1]{guindy2014atkin}.
%%%%%
\item Note that by Favard's theorem \cite[Thm.~4.4]{chihara1978intortho}, the polynomial $B_{n,r}(X)$ is also an orthogonal polynomial for the appropriate linear functional. 
Here we will discuss $B_{n,2}(X)$ as an example of this fact. To do so, we define the linear functional $\mathcal{L}^{*}:\mathbb{C}[j] \rightarrow \mathbb{C}$ as follows.
\begin{align}
&{} \mathcal{L}^{*} \left( \frac{1}{j(p) - j} \right) = \sum_{n=0}^{\infty} \mathcal{L}^{*}(j^{n}) j(p)^{-n-1} \notag \\ 
&\coloneqq \frac{-1}{393120} \left\{ \mathcal{L} \left( \frac{1}{j(p) - j} \right)^{-1} - A_{1,2}(j(p)) \right\} \label{eq:linearfunctB} \\
&= \frac{G_{14}(p)}{E_{2}(p) \Delta(p)} = \frac{{}_{2}F_{1} \left( \frac{11}{12}, \frac{19}{12} ; 3 ; \frac{1728}{j(p)} \right) }{j(p)\, {}_{2}F_{1} \left( -\frac{1}{12}, \frac{7}{12} ; 1 ; \frac{1728}{j(p)} \right) } .  \notag
\end{align}
The first few values of $\mathcal{L}^{*}(j^{n})$ are given by
\begin{align*}
&\mathcal{L}^{*}(1) =1,\, \mathcal{L}^{*}(j)=920,\, \mathcal{L}^{*}(j^{2})=1024050, \\
&\mathcal{L}^{*}(j^{3}) =1261043280,\, \mathcal{L}^{*}(j^{4}) =1653817332720.
\end{align*}
Then the polynomial $\beta_{m}(X) \coloneqq B_{m+1,2}(X)$ is the orthogonal polynomial corresponding to $\mathcal{L}^{*}$, that is, we have $\mathcal{L}^{*}\left( \beta_{m}(j) \beta_{n}(j) \right) = \mathcal{L}^{*} \left( \beta_{m}(j)^{2} \right) \delta_{m,n}$ for $m,n\ge0$. 
As a more general problem, see \cite{peherstorfer1992finite} for what kind of orthogonal polynomials occur for the ``linear fractional transformation" of the Stieltjes function as in \eqref{eq:linearfunctB}.
%%%%%
\item We have already seen that the Atkin orthogonal polynomials arise from the quasimodular solutions of the Kaneko--Zagier equation \eqref{eq:2ndKZeqn}. On the other hand, the modular solutions of the same equation can essentially be expressed by Jacobi polynomials, which are the classical  orthogonal polynomials. See \cite[\S 8]{kaneko1998supersingular} for details.
\end{enumerate}
\end{rem}

For the sake of brevity, we define some additional symbols here:
\begin{align}
\begin{split}
&{} A_{m,4}(j)\coloneqq A_{m,0}(j), \; A_{m,10}(j)\coloneqq A_{m,6}(j), \; A_{m,14}(j)\coloneqq A_{m+1,2}(j), \\
&{} N_{m,4}\coloneqq N_{m,0}, \; N_{m,10}\coloneqq N_{m,6}, \; N_{m,14}\coloneqq N_{m+1,2}. 
\end{split} \label{eq:addsymAN}
\end{align}
Under this notation, and recalling that $G_{6n+4} = E_{4}G_{6n}$, we can express the normalized extremal quasimodular forms based on the assertion (ii) of Proposition~\ref{prop:imageofAtkinL} as follows\footnote{The left-hand side of \eqref{eq:exqmfasimageofAtkinL} belongs to $\mathbb{Z}[\![p]\!]$. See Theorem 3 in \cite{nakaya2023determination}: there a stronger claim is proved.}.
\begin{align}
\begin{split}
&{} (-1)^{1-\varepsilon} N_{m, 4\delta+6\varepsilon}\, G_{12m+4\delta+6\varepsilon}(p) \\
&= \mathcal{L} \left( j^{1-\lfloor \delta/2 \rfloor} (j-1728)^{1-\varepsilon} A_{m, 4\delta+6\varepsilon}(j) \, \frac{E_{4}(p)^{\delta}E_{6}(p)^{\varepsilon}\Delta(p)^{m}}{j(p)-j} \right),
\end{split} \label{eq:exqmfasimageofAtkinL}
\end{align}
where $12m+4\delta+6\varepsilon \in 2\mathbb{Z}_{\ge 1}\backslash \{4\}$ and $m \in \mathbb{Z}_{\ge-1}, \delta \in \{0,1,2\}, \varepsilon \in \{0,1\}$. 
In addition, the orthogonality relation (iii) of Proposition \ref{prop:imageofAtkinL} can be summarized as 
\begin{align}
\mathcal{L} \left( j^{1-\lfloor \delta/2 \rfloor} (j-1728)^{1-\varepsilon} A_{m, 4\delta+6\varepsilon}(j) A_{n, 4\delta+6\varepsilon}(j) \right) = (-1)^{1-\varepsilon} N_{m,4\delta+6\varepsilon} \, \delta_{m,n} . \label{eq:orthogonalityrelation}
\end{align}
By focusing on the part $E_{4}(p)^{\delta}E_{6}(p)^{\varepsilon}\Delta(p)^{m}/(j(p)-j)$ in \eqref{eq:exqmfasimageofAtkinL}, we will derive a formula for the Fourier coefficients of the normalized extremal quasimodular forms in the next section.

%%%%%%%%%%%%%%%%%%%%%%%%%%%%%%%%%%%%%%%%%%%%%%%%%%%%%%%%%%%%%%%%%%%%%%%%%%%%%%%%%%%%%%%%%
%%%%%%%%%%%%%%%%%%%%%%%%%%%%%%%%%%%%%%%%%%%%%%%%%%%%%%%%%%%%%%%%%%%%%%%%%%%%%%%%%%%%%%%%%
%%%%%%%%%%%%%%%%%%%%%%%%%%%%%%%%%%%%%%%%%%%%%%%%%%%%%%%%%%%%%%%%%%%%%%%%%%%%%%%%%%%%%%%%%

\section{Generalized Faber polynomials and $G_{k}$} \label{sec:genFaberpoly}
We begin by reviewing certain results for weakly holomorphic modular forms in \cite{duke2008zeros} by Duke and Jenkins. 
Let $M_{k}^{!}=M_{k}^{!}(\Gamma)$ denote the vector space of all weakly holomorphic modular forms of weight $k \in 2 \mathbb{Z}$ on $\Gamma$. 
For each even integer $k$, we write $k=12m+4\delta+6\varepsilon$ with uniquely determined $m \in \mathbb{Z}, \, \delta \in \{0,1,2\}$, and $\varepsilon \in \{0,1\}$.
Under this notation, for each integer $\ell \ge -m$ there is a unique form $f_{k,\ell} \in M_{k}^{!}$ with a  Fourier expansion of the form
\begin{align}
f_{k, \ell} = f_{12m+4\delta+6\varepsilon,\, \ell}(q) = q^{-\ell} + O(q^{m+1}) . \label{eq:basefkl}
\end{align}
This form can be expressed more concretely as
\begin{align}
f_{k, \ell} = E_{4}^{\delta}E_{6}^{\varepsilon} \Delta^{m} F_{k, \ell+m}(j), \label{eq:defgenFaberpoly}
\end{align}
where $F_{k,n}(X)$ is a monic polynomial of degree $n$ with integer coefficients, and we call it the generalized Faber polynomial of weight $k$ and degree $n$ on $\Gamma$. 
Note that these forms $f_{k, \ell}$ with $\ell \ge -m$ are a basis of $M_{k}^{!}$. 
In particular, if $k\ge4$, the set $\{ f_{k,0}, f_{k,-1}, \cdots , f_{k,-m} \}$ is called the Victor Miller basis\footnote{It is sometimes called the ``reduced row echelon basis'', after the appearance of the Fourier expansion of its elements.} for $M_{k}$ (see \cite[Ch.~X, Thm.~4.4]{lang1995introductionmf}). Also, $f_{k,0}$ is called the (normalized) extremal modular form\footnote{In fact, the corresponding Faber polynomials are related to the supersingular polynomials. More specifically,  $X^{\delta}(X-1728)^{\varepsilon} F_{p-1,m}(X) \equiv ss_{p}(X) \pmod{p}$ holds under the same notation as Theorem \ref{thm:congruenceAtkin-like}. See \cite[Cor.~3]{getz2004generalizationzeros} and its proof. } of weight $k$ on $\Gamma$, which is closely related to the extremal even unimodular lattices (see \cite[pp.~33--36]{zagier2008elliptic} and \cite[p.~227]{bannai2006sphericaldesigns}).

The generating function of $f_{k, \ell}(q)$ is given as follows.
\begin{thm}[Duke--Jenkins]
For any even integer $k=12m+4\delta+6\varepsilon$, we have
\begin{align*}
\sum_{n\ge -m} f_{k,n}(q) p^{n} = \frac{f_{k,-m}(q) f_{2-k,m+1}(p)}{j(p) - j(q)},
\end{align*}
where $p$ and $q$ are independent formal variables.
\end{thm}

Rewriting this theorem using \eqref{eq:defgenFaberpoly}, we obtain the generating function of the generalized Faber polynomials:
\begin{align}
&{} \sum_{n=0}^{\infty} F_{2-(12m +4\delta +6\varepsilon ), n}(j)\, p^{n+m+1} = \frac{E_{4}(p)^{\delta}E_{6}(p)^{\varepsilon}\Delta(p)^{m}}{j(p)-j}, \label{eq:genfct2-k} \\
&{} \sum_{n=0}^{\infty} F_{12m+4\delta+6\varepsilon, n}(j)\, p^{n-m} 
= \frac{E_{4}(p)^{2-\delta}E_{6}(p)^{1-\varepsilon}}{\Delta(p)^{m+1}(j(p)-j)}. \label{eq:genfctk}
\end{align}
We give a useful "hypergeometric" formula for computing generalized Faber polynomials in Section \ref{sec:HypgenFaber}.

In equations \eqref{eq:genfct2-k} and \eqref{eq:genfctk} above, the result corresponding to $m=0, \;\delta \in \{0,1,2\},\, \varepsilon \in \{0,1\}$ has been proved by Asai, Kaneko and Ninomiya \cite[Cor.~4]{asai1997zeros}. 
Additionally, since $p\tfrac{d}{d p} j(p) = - \tfrac{E_{4}(p)^{2}E_{6}(p)}{\Delta(p)}$ holds, 
we have $-p\tfrac{d}{d p} \log (j(p)-j(q)) = \sum_{n=0}^{\infty} F_{0,n}(j(q))\, p^{n}$ and thus
\begin{align*}
\frac{1}{j(p)-j(q)} = p \cdot \exp \left( \sum_{n=1}^{\infty} F_{0,n}(j(q)) \frac{p^{n}}{n}  \right). 
\end{align*}
Note that this equation is equivalent to the denominator formula of the monster Lie algebra, which was discovered independently by Borcherds, Koike, Norton, and Zagier (see \cite{borcherds1992monstrous, borcherds2002what, asai1997zeros}): 
\begin{align*}
j(p) - j(q) = p^{-1} \prod_{m>0,\, n \in \mathbb{Z}} (1-p^{m}q^{n})^{c(m n)} ,
\end{align*}
where the exponents $c(n)$ are defined by $j(p)-744= \sum_{n\ge-1} c(n) p^{n}$. 
\begin{cor}
For any even integer $k$ and $n \in \mathbb{Z}_{\ge0}$, we have
\begin{align*}
(n+1) F_{2,n}(X) = \sum_{r=0}^{n} F_{2-k, n-r}(X) F_{k, r}(X).
\end{align*}
\end{cor}

\begin{proof}
Multiplying both sides of \eqref{eq:genfct2-k} and \eqref{eq:genfctk}, respectively, gives
\begin{align*}
&{}\sum_{n=0}^{\infty} \left(  \sum_{r=0}^{n} F_{2-12m-4\delta-6\varepsilon, n-r}(j) F_{12m+4\delta+6\varepsilon, r}(j) \right) p^{n+1} \\
&= \frac{E_{4}(p)^{2}E_{6}(p)}{\Delta(p) (j(p)-j)^{2}} = p \frac{d}{d p} \frac{1}{j(p)-j} .
\end{align*}
Since $1/(j(p)-j) = \sum_{n=0}^{\infty} F_{2,n}(j) p^{n+1}$ by \eqref{eq:genfct2-k}, the proof is completed by comparing the coefficients of $p^{n+1}$ on both sides. 
\end{proof}

\begin{thm} \label{thm:Fourieromega}
Write the even integer $k=12m+4\delta+6\varepsilon$ as above and recall the notation \eqref{eq:addsymAN}.
Let $\omega_{k,n}(r)$ and $\Omega_{2-k,n}(r)$ be the coefficients of the orthogonal polynomial expansion of the generalized Faber polynomials $F_{k,n}(X)$ and $F_{2-k,n}(X)$ by the Atkin-like polynomials $A_{r,4\delta+6\varepsilon}(X)$:
\begin{align}
F_{k,n}(X) = \sum_{r} \omega_{k,n}(r) A_{r,4\delta+6\varepsilon}(X), \quad F_{2-k,n}(X) = \sum_{r} \Omega_{2-k,n}(r) A_{r,4\delta+6\varepsilon}(X), \label{eq:AtkinexpofgenFaberpoly}
\end{align}
where the subscript $r$ runs an integer that satisfies $0 \le \deg A_{r,4\delta+6\varepsilon}(X) \le n$. 
Then, for an integer $\ell$ such that $12\ell+4\delta+6\varepsilon \in 2\mathbb{Z}_{\ge1} \backslash \{4\}$, we have
\begin{align}
&{} \sum_{n=0}^{\infty} \omega_{k,n}(\ell ) \, p^{n-m} = \frac{E_{4}(p)^{2-2\delta } E_{6}(p)^{1-2\varepsilon }}{\Delta(p)^{\ell+m+1}} \, G_{12\ell+4\delta+6\varepsilon}(p) \quad (\text{weight $2-k$}) ,  \label{eq:omegakasFouriercoeff} \\
&{} \sum_{n=0}^{\infty} \Omega_{2-k,n}(\ell ) \, p^{n+m+1} = \Delta(p)^{m-\ell} G_{12\ell+4\delta+6\varepsilon}(p) \quad (\text{weight $k$}) .  \label{eq:omega2-kasFouriercoeff}
\end{align}
Particularly, we have $G_{12\ell+4\delta+6\varepsilon}(p)=\sum_{n=0}^{\infty} \Omega_{2-(12\ell +4\delta +6\varepsilon ),n}(\ell )\, p^{n+\ell+1}$.
\end{thm}

\begin{proof}
Combining the first equation of \eqref{eq:AtkinexpofgenFaberpoly} and the orthogonality relation \eqref{eq:orthogonalityrelation}, we have
\begin{align*}
&{} \mathcal{L} \left( j^{1-\lfloor \delta/2 \rfloor} (j-1728)^{1-\varepsilon} A_{\ell, 4\delta+6\varepsilon}(j) F_{k,\, n}(j) \right) \\
&= \sum_{r} \omega_{k,n}(r) \mathcal{L} \left( j^{1-\lfloor \delta/2 \rfloor} (j-1728)^{1-\varepsilon} A_{\ell, 4\delta+6\varepsilon}(j) A_{r, 4\delta+6\varepsilon}(j) \right) \\
&= (-1)^{1-\varepsilon} N_{\ell, 4\delta+6\varepsilon}\, \omega_{k,n}(\ell ) .
\end{align*}
On the other hand, by using the generating function \eqref{eq:genfctk} of the generalized Faber polynomial and equation \eqref{eq:exqmfasimageofAtkinL}, we obtain
\begin{align*}
&{} \sum_{n=0}^{\infty} \mathcal{L} \left( j^{1-\lfloor \delta/2 \rfloor} (j-1728)^{1-\varepsilon} A_{\ell, 4\delta+6\varepsilon}(j) F_{k,\, n}(j) \right) p^{n-m}  \\
&= \mathcal{L} \left( j^{1-\lfloor \delta/2 \rfloor} (j-1728)^{1-\varepsilon} A_{\ell, 4\delta+6\varepsilon}(j) \sum_{n=0}^{\infty} F_{k,\, n}(j) p^{n-m} \right) \\
&= \mathcal{L} \left( j^{1-\lfloor \delta/2 \rfloor} (j-1728)^{1-\varepsilon} A_{\ell, 4\delta+6\varepsilon}(j) \frac{E_{4}(p)^{2-\delta}E_{6}(p)^{1-\varepsilon}}{\Delta(p)^{m+1}(j(p)-j)} \right) \\
&= \frac{E_{4}(p)^{2-2\delta}E_{6}(p)^{1-2\varepsilon}}{\Delta(p)^{\ell+m+1}} \; \mathcal{L} \left( j^{1-\lfloor \delta/2 \rfloor} (j-1728)^{1-\varepsilon} A_{\ell, 4\delta+6\varepsilon}(j) \frac{E_{4}(p)^{\delta}E_{6}(p)^{\varepsilon}\Delta(p)^{\ell}}{j(p)-j} \right) \\
&=\frac{E_{4}(p)^{2-2\delta}E_{6}(p)^{1-2\varepsilon}}{\Delta(p)^{\ell+m+1}} \; (-1)^{1-\varepsilon} N_{\ell, 4\delta+6\varepsilon}\, G_{12\ell+4\delta+6\varepsilon}^{(1)}(p).
\end{align*}
This gives \eqref{eq:omegakasFouriercoeff}, and \eqref{eq:omega2-kasFouriercoeff} is also obtained from a similar calculation. 
\end{proof}

\begin{rem}
Specializing to $k=2=12 \cdot (-1) +4 \cdot 2 +6 \cdot 1$ and $\ell=-1$ in \eqref{eq:omegakasFouriercoeff}, we obtain the Fourier coefficients of the moment-generating function \eqref{eq:Stieltjes function} as $\omega_{2,n}(-1)$:
\begin{align*}
&{} \sum_{n=0}^{\infty} \omega_{2,n}(-1)\, p^{n+1} = \frac{E_{4}(p)^{-2}E_{6}(p)^{-1}}{\Delta(p)^{-1}} G_{2}(p) \\
&= \frac{E_{2}(p)E_{4}(p)}{j(p)E_{6}(p)} = p -24 p^{2} +196812 p^{3} +38262208 p^{4} +O(p^{5}) .
\end{align*}
Here, $\omega_{2,n}(-1)$ is the ``constant term'' in the expansion of  $F_{2,n}(X)$ by the Atkin polynomials.
\begin{align*}
F_{2,n}(X) &= \sum_{r} \omega_{2,n}(r) A_{r,14}(X) = \sum_{r} \omega_{2,n}(r) A_{r+1,2}(X) \\
&= \omega_{2,n}(-1)A_{0,2}(X) + \omega_{2,n}(0)A_{1,2}(X) + \cdots + \omega_{2,n}(n-1)A_{n,2}(X).
\end{align*}
\end{rem}

\begin{ex}
Case of $k=14$, i.e., $(m,\delta,\varepsilon)=(0,2,1)$ and $\ell \ge  -1$. Then the table of $\omega_{14,n}(\ell )$ is as follows.
\begin{table}[H]
\caption{The first few values of coefficients $\omega_{14,n}(\ell )$.}
\label{tab:o14nl}
\begin{center}
\begin{tabular}{|c|ccccc|}\hline
$n \backslash \ell $ & $-1$ & 0 & 1 & 2 & 3 \\ \hline
0 & 1 & 0 & 0 & 0 & 0 \\ 
1 & 0 & 1 & 0 & 0 & 0 \\
2 & 196560 & 176 & 1 & 0 & 0 \\
3 & 42981120 & 208302 & $\frac{1536}{5}$ & 1 & 0 \\[2pt]
4 & 41292342000 & 78071008 & $\frac{1176672}{5}$ & 432 & 1 \\[2pt] \hline
\end{tabular}
\end{center}
\end{table}
The rows of Table \ref{tab:o14nl} are obtained from an orthogonal polynomial expansion, e.g., 
\begin{align*}
F_{14,0}(X) &= 1 =A_{0,2}(X), \\
F_{14,1}(X) &= X-720 = 0 \cdot A_{0,2}(X) +A_{1,2}(X) , \\
F_{14,2}(X) &= X^{2}-1464 X+339120 = 196560A_{0,2}(X) +176 A_{1,2}(X) + A_{2,2}(X) .
\end{align*}
The columns of Table \ref{tab:o14nl} are the Fourier coefficients of certain forms closely related to the normalized extremal quasimodular forms, from \eqref{eq:omegakasFouriercoeff} in Theorem \ref{thm:Fourieromega}.
\begin{align*}
&{}  \frac{G_{2}(p)}{E_{4}(p)^{2} E_{6}(p)} = 1+196560 p^2+42981120 p^3+41292342000 p^4+O\left(p^5\right) , \\
&{} \frac{G_{14}(p)}{E_{4}(p)^{2} E_{6}(p) \Delta(p)} = p+176 p^2+208302 p^3+78071008 p^4+O\left(p^5\right) , \\
&{} \frac{G_{26}(p)}{E_{4}(p)^{2} E_{6}(p) \Delta(p)^{2}} = p^2+\frac{1536 p^3}{5}+\frac{1176672 p^4}{5}+\frac{531453184 p^5}{5} +O\left(p^6\right) .
\end{align*}
Note that since we have already obtained the hypergeometric expressions of $G_{k}$ in Proposition~\ref{prop:Gto2F1}, we can calculate a certain $\omega_{k,n}(\ell )$. For example, the subdiagonal elements of Table \ref{tab:o14nl} are  
\begin{align*}
\omega_{14,n}(n-2) &= \frac{48 (n-1) (5 n+1)}{2 n-1} \quad (n\ge1)
\end{align*}
and the subsubdiagonal elements are
\begin{align*}
\omega_{14,n}(n-3) &= \frac{36 \left(400 n^4-2210 n^3+14931 n^2-29408 n+15832\right)}{(n-1) (2 n-3)} \quad (n\ge2).
\end{align*}
\end{ex}

\begin{cor}
For any even integer $k=12m+4\delta+6\varepsilon$, $n \in \mathbb{Z}_{\ge0}$ and integers $\ell$ and $\ell'$ such that $12\ell+4\delta+6\varepsilon, 12\ell' +14-4\delta -6\varepsilon \in 2\mathbb{Z}_{\ge1} \backslash \{4\}$, we have
\begin{align}
&{} \sum_{d=0}^{n} \omega_{k,d}(\ell ) F_{2-k,n-d}(X) = \sum_{d=0}^{n} \Omega_{2-k,n-d}(\ell ) F_{k,d}(X) , \label{eq:oFOF} \\
&{} \sum_{d=0}^{n} \omega_{k,d}(\ell ) \omega_{2-k,n-d}(\ell' ) = \sum_{d=0}^{n} \Omega_{2-k,n-d}(\ell ) \Omega_{k,d}(\ell' ).  \label{eq:ooOO}
\end{align}
\end{cor}

\begin{proof}
From equations \eqref{eq:omegakasFouriercoeff} and \eqref{eq:omega2-kasFouriercoeff} we find that $G_{12\ell+4\delta+6\varepsilon}(p)$ has the following two expressions:
\begin{align*}
G_{12\ell+4\delta+6\varepsilon}(p) &= E_{4}(p)^{2\delta -2} E_{6}(p)^{2\varepsilon -1} \Delta(p)^{m+\ell+1} \sum_{n=0}^{\infty} \omega_{k,n}(\ell ) \, p^{n-m}, \\
G_{12\ell+4\delta+6\varepsilon}(p) &= \Delta(p)^{\ell -m} \sum_{n=0}^{\infty} \Omega_{2-k,n}(\ell ) \, p^{n+m+1},
\end{align*}
and hence
\begin{align*}
 E_{4}(p)^{\delta} E_{6}(p)^{\varepsilon } \Delta(p)^{m} \sum_{n=0}^{\infty} \omega_{k}(n,\ell) \, p^{n-m} = \frac{E_{4}(p)^{2-\delta} E_{6}(p)^{1-\varepsilon }}{\Delta(p)^{m+1}} \sum_{n=0}^{\infty} \Omega_{2-k}(n,\ell) \, p^{n+m+1}  .
\end{align*}
Dividing both sides of this equation by $(j(p)-j)$ and using equations \eqref{eq:genfct2-k} and \eqref{eq:genfctk}, we have
\begin{align*}
&{}\left( \sum_{n=0}^{\infty} F_{2-k, n}(j)\, p^{n+m+1}  \right) \left( \sum_{n=0}^{\infty} \omega_{k,n}(\ell ) \, p^{n-m} \right) \\
&= \left( \sum_{n=0}^{\infty} F_{k, n}(j)\, p^{n-m} \right) \left( \sum_{n=0}^{\infty} \Omega_{2-k,n}(\ell ) \, p^{n+m+1} \right)  .
\end{align*}
The desired result \eqref{eq:oFOF} can be obtained by comparing the coefficients of $p^{n}$ on both sides of the above equation.

Since $2-k=12(-m-1)+4(2-\delta)+6(1-\varepsilon)$, equations \eqref{eq:omegakasFouriercoeff} and \eqref{eq:omega2-kasFouriercoeff} can be rewritten as follows:
\begin{align*}
&{} \sum_{n=0}^{\infty} \omega_{2-k,n}(\ell' ) \, p^{n+m+1} = \frac{E_{4}(p)^{2\delta -2} E_{6}(p)^{2\varepsilon -1}}{\Delta(p)^{\ell' -m}} \, G_{12\ell' +14-4\delta -6\varepsilon}(p), \\
&{} \sum_{n=0}^{\infty} \Omega_{k,n}(\ell' ) \, p^{n-m} = \Delta(p)^{-m-1-\ell'} G_{12\ell' +14-4\delta -6\varepsilon}(p),
\end{align*}
where the integer $\ell'$ satisfies $12\ell' +14-4\delta -6\varepsilon \in 2\mathbb{Z}_{\ge1} \backslash \{4\}$.
Thus we have
\begin{align*}
&{} \sum_{n=0}^{\infty} \left( \sum_{d=0}^{n} \omega_{k,d}(\ell ) \omega_{2-k,n-d}(\ell' ) \right) p^{n+1} = \sum_{n=0}^{\infty} \left( \sum_{d=0}^{n} \Omega_{2-k,n-d}(\ell ) \Omega_{k,d}(\ell' ) \right) p^{n+1} \\
&= \frac{G_{12\ell +4\delta +6\varepsilon }(p) G_{12\ell' +14-4\delta -6\varepsilon}(p)}{\Delta(p)^{\ell +\ell' -1}}. 
\end{align*}
The desired result \eqref{eq:ooOO} can be obtained by comparing the coefficients of $p^{n+1}$ of the above equation, which are also Fourier coefficients of certain forms.
\end{proof}

%%%%%%%%%%%%%%%%%%%%%%%%%%%%%%%%%%%%%%%%%%%%%%%%%%%%%%%%%%%%%%%%%%%%%%%%%%%%%%%%%%%%%%%%%
%%%%%%%%%%%%%%%%%%%%%%%%%%%%%%%%%%%%%%%%%%%%%%%%%%%%%%%%%%%%%%%%%%%%%%%%%%%%%%%%%%%%%%%%%
%%%%%%%%%%%%%%%%%%%%%%%%%%%%%%%%%%%%%%%%%%%%%%%%%%%%%%%%%%%%%%%%%%%%%%%%%%%%%%%%%%%%%%%%%

\section{Atkin inner product as a weight zero counterpart of the Petersson inner product} \label{sec:Atkin-Petersson}

In \cite[Thm.~9.2]{borcherds1998automorphic} Borcherds showed\footnote{The case $n=0$ in Theorem 9.2 coincides with the definition \eqref{def:Atkininnerprod} of the Atkin inner product, but this fact is not mentioned in \cite{borcherds1998automorphic}. For the proof, Borcherds refers to \cite[p.~102]{lerche1989lattices} by physicists.} that the Atkin inner product can be expressed as the following integral (here $f(\tau), g(\tau) \in \mathbb{C}[j(\tau)]$):
\begin{align}
(f(\tau),g(\tau)) = \frac{1}{\mathrm{Vol}(\Gamma \backslash \mathfrak{H})} \lim_{T \rightarrow \infty} \int_{\Omega_{T}} f(\tau) g(\tau) \frac{d x d y}{y^{2}} \quad (\tau = x + i y), \label{eq:Atkin-Peterssonprod}
\end{align}
where $\Omega_{T} = \{ \tau \in \mathfrak{H} \mid  |\tau| \ge 1, |x|\le \tfrac{1}{2}, y \le T \}$ is the truncation of the fundamental domain $\Gamma \backslash \mathfrak{H}$ and $\mathrm{Vol}(\Gamma \backslash \mathfrak{H}) = \int_{\Gamma \backslash \mathfrak{H}} y^{-2} dxdy =\tfrac{\pi}{3} $ is the hyperbolic volume of $\Gamma \backslash \mathfrak{H}$. 
From this expression, the Atkin inner product can be regarded as a weight zero counterpart of the Petersson inner product $\langle \cdot , \cdot \rangle$ defined by
\begin{align*}
\langle  f(\tau) , g(\tau) \rangle = \int_{\Gamma \backslash \mathfrak{H}} f(\tau) \overline{g(\tau)}\, y^{k} \frac{d x d y}{y^{2}}  \quad (\tau = x + i y).
\end{align*} 
Here, $f(\tau), g(\tau) \in M_{k}$ and either $f(\tau)$ or $g(\tau)$ belongs to $S_{k}$. While the Petersson inner product is Hermitian, the equation \eqref{eq:Atkin-Peterssonprod} shows that the Atkin inner product is not: In particular, $\langle  f , g \rangle = \overline{ \langle g , f \rangle }$  for $fg \in S_{2k}$ but not necessarily $(f,g) = \overline{(g,f)}$ for general $f,g \in \mathbb{C}[j(\tau)]$. 

We define the $n$-th Hecke operator $T_{n}\; (n\in \mathbb{Z}_{\ge1})$ on a modular form $f(\tau )=\sum_{m} c(m) q^{m}$ of weight $k$ on $\Gamma$ as
\begin{align*}
( f |_{k} T_{n})(\tau ) = n^{k-1} \sum_{\substack{ad=n, \; d>0 \\ 0 \le b \le d-1}} d^{-k} f \left( \frac{a\tau +b}{d} \right) = \sum_{m} \left( \sum_{0<d | (m,n)} d^{k-1} c \left(\frac{m n}{d^{2}} \right) \right) q^{m}.
\end{align*}
Note that the Hecke operators $T_{n}$ are self-adjoint with respect to these inner products, i.e., $\langle f|_{k}T_{n} , g \rangle = \langle f , g|_{k}T_{n} \rangle$ and $(f|_{0}T_{n} , g)=(f , g|_{0}T_{n})$ hold. 
For a proof of the latter, see \cite[Thm.~2, \S 5]{kaneko1998supersingular}, although the definition of the Hecke operator is slightly different from ours, it is easy to modify.

The $n$-th Poincar\'{e} series of even weight $k \ge 4$ on $\Gamma$ is defined by 
\begin{align}
P_{k,n}(\tau ) = \frac{1}{2} \sum_{\substack{c,d \in \mathbb{Z} \\ (c,d)=1}} (c\tau +d)^{-k} \exp \left( 2 \pi i n \, \frac{a \tau +b}{c \tau +d} \right), \label{eq:Poincareseries}
\end{align}
where the integers $a$ and $b$ are chosen so that $\bigl( \begin{smallmatrix} a & b \\ c & d \end{smallmatrix} \bigr) \in \Gamma$. 
In particular, when $n=0$, we obtain $P_{k,0}(\tau) = E_{k}(\tau)$, which is the usual Eisenstein series. 
It is well-known that, for $n \in \mathbb{Z}_{\ge 1}$, the space of cusp forms $S_{k}$ is spanned by $\{ P_{k,n}(\tau) \}_{1\le n \le [k/12]}$, also $P_{k,n}(\tau) = n^{-k+1} (P_{k,1} |_{k} T_{n})(\tau)$ holds. 
%For $n<0$ we have $P_{k,n}(\tau) \in M_{k}^{!}(\Gamma)$.

The right-hand side of \eqref{eq:Poincareseries} does not converge for $k=0$, but the following result \cite[Eq.~(18)]{knopp1990rademacher} by Knopp is known, which is a rewrite of Rademacher's result  \cite[Eq.~(4.1)]{rademacher1939absolutej}:
\begin{align*}
j(\tau ) -732 = \frac{1}{2} \lim_{K \rightarrow \infty} \sum_{|c| \le K} \sum_{\substack{|d| \le K \\ (c,d)=1}} \left\{ \exp \left( -2 \pi i \, \frac{a \tau +b}{c \tau +d} \right) - \exp \left( - 2 \pi i \frac{a}{c} \right) \right\},
\end{align*}
where the integers $a$ and $b$ are chosen as in \eqref{eq:Poincareseries}, and if $c=0$, we assume that the term $\exp \left( - 2 \pi i \tfrac{a}{c} \right)=0$. Knopp states that this representation allows $j(\tau )$ to be considered as a parabolic Poincar\'{e} series of weight $0$ on $\Gamma$.

For convenience, we now consider the function $H_{1}(j(\tau ))=j(\tau )-720$, which differs from $j(\tau )-732$ by an additive constant\footnote{If the constant term $c_{0}$ is defined by (7.72) in \cite{rademacher1938fourierj}, the right-hand side of (7.71) in the same paper is equal to $j(\tau)-720=H_{1}(j(\tau))=A_{1}(j(\tau))$. See the discussion in \cite[\S 8]{rademacher1938fourierj}.}, to be the Poincar\'{e} series of weight 0 on $\Gamma$ (we set $H_{0}(j)\coloneqq1$).
Furthermore, following known result for $P_{k,n}(\tau)$ mentioned above, we also define
\begin{align*}
H_{n}(j(\tau)) \coloneqq n (H_{1}|_{0}T_{n})(\tau ) = n (j(\tau)-720) |_{0} T_{n} \quad (n\ge1). 
\end{align*}
In \cite[Problem~2]{kaneko1999zeros}, Kaneko presented the following problem with this ``Poincar\'{e}  series'' $H_{n}(j)$ and the original Atkin polynomial $A_{\ell}(j)$: 
\begin{quote}
\textit{What is the value of the Atkin inner product $(H_{n}(j),A_{\ell }(j))$? In particular, is the value always non-negative?}
\end{quote}
We cannot answer the question of non-negativity, but the value of $(H_{n}(j),A_{\ell }(j))$ can be interpreted as an integral Fourier coefficient of some forms that are closely related to the extremal quasimodular forms of depth 1 on $\Gamma$.
\begin{prop} \label{prop:AtkininnerprodHandA}
For any integer $\ell \ge 0$, we have
\begin{align*}
\sum_{n=0}^{\infty} (H_{n}(j), A_{\ell}(j)) p^{n} =
\begin{cases}
1 & \text{ if } \ell=0, \\
N_{\ell,2} \, G_{12\ell+2}(p) \Delta (p)^{-\ell} = N_{\ell,2}\, p^{\ell }(1+O(p)) \in \mathbb{Z}[\![p]\!] & \text{ if } \ell \ge 1.
\end{cases}
\end{align*}
\end{prop}

\begin{proof}
Since the generalized Faber polynomial $F_{0,n}(j)$ satisfies $F_{0,n}(j(\tau))= n (j(\tau)-744) |_{0} T_{n}$ for $n\ge1$ (see \cite{asai1997zeros}), $H_{n}(j)$ and $F_{0,n}(j)$ differ only in the constant term:
\begin{align*}
H_{n}(j)= F_{0,n}(j) + 24 \sum_{\substack{ad=n, \; d>0 \\ 0 \le b \le d-1}} 1 = F_{0,n}(j) + 24 \sigma_{1}(n) \quad (n\ge1).
\end{align*} 
By setting $k=2=12\cdot (-1)+4\cdot2+6\cdot1$ in the second equation of  \eqref{eq:AtkinexpofgenFaberpoly}, we have $F_{0,n}(j) = \sum_{r} \Omega_{0,n}(r) A_{r,14}(j)$ and hence 
\begin{align*}
(H_{n}(j), A_{\ell}(j)) &= (  F_{0,n}(j) + 24 \sigma_{1}(n), A_{\ell-1,14}(j)) \\
&= N_{\ell-1,14} \, \Omega_{0,n}(\ell -1) + 24 \sigma_{1}(n) \, (1, A_{\ell -1,14}(j)) \quad \text{(by \eqref{eq:orthogonalityrelation})}\\
&= N_{\ell, 2} \, \Omega_{0,n}(\ell -1) + 24 \sigma_{1}(n) \, (1, A_{\ell}(j)) \quad (n\ge1).
\end{align*}
(From the notation $A_{m,14}(X)=A_{m+1,2}(X)$, this is equal to the original Atkin polynomial $A_{m+1}(X)$.) On the other hand, by setting $k=2=12\cdot (-1)+4\cdot2+6\cdot1$ in \eqref{eq:omega2-kasFouriercoeff} of Theorem \ref{thm:Fourieromega}, we have $\sum_{n=0}^{\infty} \Omega_{0,n}(\ell) p^{n} = \Delta (p)^{-1-\ell} G_{12\ell +14}(p)$. Replacing $\ell$ in this equation with $\ell-1$, we obtain
\begin{align*}
\sum_{n=0}^{\infty} (H_{n}(j), A_{\ell}(j)) p^{n} &= (1,A_{\ell}(j)) + \sum_{n=1}^{\infty} \{ N_{\ell, 2} \, \Omega_{0,n}(\ell -1) + 24 \sigma_{1}(n) \, (1, A_{\ell}(j))  \} p^{n} \\
&= N_{\ell,2}\, G_{12\ell +2}(p) \Delta (p)^{-\ell} +24 (1,A_{\ell}(j)) \sum_{n=1}^{\infty} \sigma_{1}(n) p^{n}.
\end{align*}
Since the orthogonality relation for Atkin polynomials, the second term of  the right-hand side of the above equation vanishes for $\ell \ge1$, and when $\ell=0$, only the constant term survives. 
Finally, since $N_{\ell ,2}G_{12\ell +2}(p) \in p^{2\ell }\mathbb{Z}[\![p]\!]$ from Theorem~3 in \cite{nakaya2023determination} and $\Delta (p)^{-\ell } \in p^{-\ell } \mathbb{Z}[\![p]\!]$ is obvious, so the desired integrality is concluded.
\end{proof}

\begin{rem}
\begin{enumerate}[(i)]
\item Using the self-adjointness of the Hecke operator, the claim for $\ell=0$ can also be proved as follows: for $n\ge1$ we have
\begin{align*}
(H_{n}(j),1) = (n(j-720)|_{0}T_{n},1) = (j-720, n|_{0}T_{n}) = \sigma_{1}(n) (j-720,1) =0.
\end{align*}
\item From Proposition \ref{prop:Gto2F1}, the form $\sum_{n=0}^{\infty} (H_{n}(j), A_{\ell}(j)) p^{n}$ for $\ell \ge1$ can be expressed using hypergeometric series.  
\end{enumerate}
\end{rem}

We define the regularized Petersson inner product, which differs from equation \eqref{eq:Atkin-Peterssonprod} only in the complex conjugate symbol and the constant multiple, as follows (here $f(\tau), g(\tau) \in \mathbb{C}[j(\tau)]$): 
\begin{align*}
\langle f(\tau) , g(\tau) \rangle_{\mathrm{reg}} = \lim_{T \rightarrow \infty} \int_{\Omega_{T}} f(\tau) \overline{g(\tau)} \frac{d x d y}{y^{2}} \quad (\tau = x + i y).
\end{align*}
In \cite[Thm.~1]{duke2016regularized}, Duke, Imamo\={g}lu and T\'{o}th showed that an analytic formula for the value of the inner product $\langle H_{m}(j(\tau)) , H_{n}(j(\tau)) \rangle_{\mathrm{reg}}$ for different $m,n \in \mathbb{Z}_{\ge0}$ using the Kloosterman sum and the Bessel functions. (The correspondence between their symbol and ours is $ f_{m}=H_{m}$ except for $f_{0}=0$.) For the case of $m=n$, see \cite[Thm.~1.2]{bringmann2017regularized}. 
Note that this regularized Petersson inner product and the Atkin inner product \eqref{eq:Atkin-Peterssonprod} look very similar, but the values of the two inner products $\langle H_{m}(j(\tau)) , H_{n}(j(\tau)) \rangle_{\mathrm{reg}}$ and $\tfrac{\pi}{3}( H_{m}(j) , H_{n}(j))$ are different because $\overline{H_{n}(j(\tau))}$ and $H_{n}(j(\tau))$ are different functions as complex functions. 
For example, $\langle H_{2}, H_{1} \rangle_{\mathrm{reg}} =366.765$ in  the numerical example in \cite{duke2016regularized}, but $\tfrac{\pi}{3}(H_{2},H_{1})=6.25745 \dots \times 10^{7}$ in our calculation. 
We also note that $\langle f g , 1 \rangle_{\mathrm{reg}}=\tfrac{\pi}{3}(f,g)$ holds, and from this representation, the regularized Petersson inner product can be regarded as a generalization of the Atkin inner product.

Inspired by the inner product $\langle H_{m}(j(\tau)) , H_{n}(j(\tau)) \rangle_{\mathrm{reg}}$, let us consider the generating series of the Atkin inner product $( H_{m}(j) , H_{n}(j))$.
\begin{cor}\label{cor:HH}
\begin{enumerate}[(i)]
\item For any integer $\ell \ge1$, the following equation holds. 
\begin{align*}
\sum_{n=1}^{\infty} (H_{n}(j), H_{\ell}(j)) p^{n} = \sum_{r=0}^{\ell -1} \Omega_{0,\ell }(r) \, N_{r+1,2} \, G_{12r+14}(p) \Delta(p)^{-r-1}. 
\end{align*}
\item The following equation holds as a formal power series in $\mathbb{Z}[\![p,q]\!]$.
\begin{align*}
\sum_{n, \ell  \ge 1} (H_{n}(j), H_{\ell}(j)) p^{n}q^{\ell} = \sum_{r=0}^{\infty} N_{r+1,2}\, G_{12r+14}(p) G_{12r+14}(q) \bigl( \Delta(p) \Delta(q) \bigr)^{-r-1} .
\end{align*}
\end{enumerate}
\end{cor}

\begin{proof}
By expanding the polynomial $H_{\ell}(j)=F_{0,\ell}(j) +24 \sigma_{1}(\ell)$ with the Atkin polynomial, as in Proposition \ref{prop:AtkininnerprodHandA}, we have
\begin{align*}
&{} \sum_{n=1}^{\infty} (H_{n}(j), H_{\ell}(j)) p^{n} = \sum_{n=1}^{\infty} \left\{ \left( H_{n}(j), \sum_{r=-1}^{\ell -1} \Omega_{0,\ell }(r) A_{r,14}(j) \right) + 24 \sigma_{1}(\ell) (H_{n}(j),1)  \right\} p^{n} \\
&= \sum_{n=1}^{\infty} \left\{ \sum_{r=-1}^{\ell -1} \Omega_{0,\ell }(r) (H_{n}(j) , A_{r+1}(j) ) \right\} p^{n} = \sum_{r=0}^{\ell -1} \Omega_{0,\ell }(r) \sum_{n=1}^{\infty} (H_{n}(j), A_{r+1}(j) ) p^{n} .
\end{align*}
(Note that $(H_{n}(j),1)=0$ for $n \ge 1$.) Then, by applying Proposition \ref{prop:AtkininnerprodHandA}, we obtain the desired result of (i). Furthermore, we have 
\begin{align*}
&{} \sum_{n, \ell  \ge 1} (H_{n}(j), H_{\ell}(j)) p^{n}q^{\ell} = \sum_{\ell \ge 1} \left\{  \sum_{r=0}^{\ell -1} \Omega_{0,\ell }(r) \, N_{r+1,2} \, G_{12r+14}(p) \Delta(p)^{-r-1} \right\} q^{\ell} \\
&= \sum_{r=0}^{\infty} N_{r+1,2} \, G_{12r+14}(p) \Delta(p)^{-r-1} \sum_{\ell \ge 1} \Omega_{0,\ell }(r) \, q^{\ell}.
\end{align*}
With the help of Theorem \ref{thm:Fourieromega}, the sum for $\ell$ in the above equation is equal to $G_{12r+14}(q) \Delta(q)^{-r-1}$, and thus we obtain the second assertion.
\end{proof}

\begin{ex}
The first few coefficients $\Omega_{0,\ell }(r)$ are given in the following table. We emphasize again that the columns of this table are the Fourier coefficients of some forms from Theorem \ref{thm:Fourieromega}. 
\begin{table}[H]
\caption{The first few values of coefficients $\Omega_{0,\ell}(r)$.}
\begin{center}
\begin{tabular}{|c|ccccc|}\hline
$\ell \backslash r$ & 0 & 1 & 2 & 3 & 4 \\ \hline
1 & 1 & 0 & 0 & 0 & 0 \\ 
2 & 152 & 1 & 0 & 0 & 0 \\
3 & 7446 & $\frac{1416}{5}$ & 1 & 0 & 0 \\[2pt]
4 & 200752 & $\frac{156648}{5}$ & 408 & 1 & 0 \\[2pt]
5 & 3685870 & $\frac{9867424}{5}$ & 70479 & $\frac{1592}{3}$ & 1 \\[2pt] \hline
\end{tabular}
\end{center}
\end{table}
By setting $\ell=1,2$ in (i) of Corollary \ref{cor:HH}, we have 
\begin{align*}
&{} \sum_{n=1}^{\infty} (H_{n}(j), H_{1}(j)) p^{n} = 1 \cdot \frac{N_{1,2}\, G_{14}(p)}{\Delta(p)} \\
&= 393120 p+59754240 p^{2} +2927171520 p^{3} +78919626240 p^{4} +O\left(p^{5}\right), \\
&{} \sum_{n=1}^{\infty} (H_{n}(j), H_{2}(j)) p^{n} = 152 \cdot \frac{N_{1,2}\, G_{14}(p)}{\Delta(p)} + 1 \cdot \frac{N_{2,2}\, G_{26}(p)}{\Delta(p)^{2}} \\
&= 59754240 p +78920412480 p^{2} +20222985968640 p^{3}  +O\left(p^{4} \right) .
\end{align*}
\end{ex}

The generating series of values of the Atkin inner product of generalized Faber polynomials of weight $0$ can be obtained by the same calculation as in Proposition~\ref{prop:AtkininnerprodHandA} and Corollary~\ref{cor:HH} (recall that $N_{0,2}\coloneqq1$):
\begin{align}
\sum_{n, \ell  \ge 0} (F_{0,n}(j), F_{0,\ell}(j)) p^{n}q^{\ell} = \sum_{r=0}^{\infty} N_{r,2}\, G_{12r+2}(p) G_{12r+2}(q) \bigl( \Delta(p) \Delta(q) \bigr)^{-r} .  \label{eq:FFpqNGG}
\end{align}
In fact, the left-hand side of the above equation has a simple closed form.
\begin{thm} \label{thm:AtkinFaberFaber}
The following equation holds as a formal power series in $\mathbb{Z}[\![p,q]\!]$.
\begin{align*}
\sum_{n, \ell  \ge 0} (F_{0,n}(j), F_{0,\ell}(j)) p^{n}q^{\ell} = \frac{\psi (p,q) - \psi (q,p)}{j(p)^{-1} - j(q)^{-1}}, \quad \psi (p,q) \coloneqq  \frac{E_{2}(p) E_{6}(q)}{j(p) E_{4}(q)}. 
\end{align*}
\end{thm}
The integrality follows from the fact that both $\psi (p,q) - \psi (q,p)$ and $j(p)^{-1} - j(q)^{-1}$ have expansions of the form $(p-q)(1+ \cdots ) \in \mathbb{Z}[\![p,q]\!]$. 
Alternatively, we know from equation \eqref{eq:genfctk} that $F_{0,n}(j) \in \mathbb{Z}[j]$, and furthermore, since $\mathcal{L}(j^{m}) \in \mathbb{Z}$, we see that $(F_{0,n}(j), F_{0,\ell}(j)) \in \mathbb{Z}$.

We now give two different proofs of Theorem \ref{thm:AtkinFaberFaber}. The first proof is simple and straightforward, and the second proof shows that the right-hand side of Theorem~\ref{thm:AtkinFaberFaber} is equal to the right-hand side of \eqref{eq:FFpqNGG}.

\begin{proof}[Direct proof of Theorem \ref{thm:AtkinFaberFaber}]
By setting $m=\delta =\varepsilon =0$ in \eqref{eq:genfctk}, we have
\begin{align*}
\sum_{n=0}^{\infty} F_{0,n}(j)p^{n} = \frac{j(p)E_{6}(p)}{E_{4}(p)(j(p)-j)}
\end{align*}
and so
\begin{align*}
&{} \sum_{n,\ell \ge0} F_{0,n}(j) F_{0,\ell }(j)p^{n}q^{\ell} = \frac{j(p)j(q)E_{6}(p)E_{6}(q)}{E_{4}(p)E_{4}(q)(j(p)-j)(j(q)-j)} \\
&= \frac{j(p)j(q)E_{6}(p)E_{6}(q)}{E_{4}(p)E_{4}(q)(j(q)-j(p))} \left( \frac{1}{j(p)-j} - \frac{1}{j(q)-j} \right).
\end{align*}
By acting on both sides of this equation with the linear functional $\mathcal{L}$ corresponding to the Atkin inner product (recall \eqref{eq:momentL}), we obtain
\begin{align*}
&{}\sum_{n,\ell \ge0} \mathcal{L}(F_{0,n}(j) F_{0,\ell }(j))p^{n}q^{\ell} \\
&= \frac{j(p)j(q)E_{6}(p)E_{6}(q)}{E_{4}(p)E_{4}(q)(j(q)-j(p))} \left\{ \mathcal{L}\left( \frac{1}{j(p)-j}\right) - \mathcal{L}\left(\frac{1}{j(q)-j} \right) \right\}.
\end{align*}
From equation \eqref{eq:Stieltjes function}, the moment-generating function can be rewritten in the form using the Eisenstein series and the elliptic modular invariant to obtain the desired representation. 
\end{proof}

Performing the same calculations  for \eqref{eq:genfctk} as in the proof above, we obtain the following result for the generating series of $(F_{k,n}(j), F_{k,\ell}(j))$.
\begin{cor}
Write the even integer $k=12m+4\delta+6\varepsilon$ as in Section \ref{sec:genFaberpoly}. Then the following equation holds as a formal power series in $\mathbb{Z}[\![p,q]\!]$.
\begin{align*}
\sum_{n, \ell  \ge 0} (F_{k,n}(j), F_{k,\ell}(j)) p^{n-m}q^{\ell -m} = \frac{\sum_{n, \ell  \ge 0} (F_{0,n}(j), F_{0,\ell}(j)) p^{n}q^{\ell}}{(E_{4}(p)E_{4}(q))^{\delta} (E_{6}(p)E_{6}(q))^{\varepsilon } (\Delta (p)\Delta (q))^{m}} .
\end{align*}
\end{cor}

To give the second ``hypergeometric" proof of Theorem \ref{thm:AtkinFaberFaber}, we introduce Rogers' results on the continued fraction expansion of a formal power series \cite[\S 5]{rogers1907representation}. For the sake of brevity, we define some symbols here. 
For a formal power series $h(z)=\sum_{n=0}^{\infty} a_{n} z^{n}$, we define the formal Borel transform $B(h)$ and the formal Laplace transform $L(h)$ of $h(z)$ as follows:
\begin{align*}
B(h)(\xi) = \sum_{n=0}^{\infty} a_{n} \frac{\xi^{n}}{n!}, \quad L(h)(s) = \int_{0}^{\infty} h(z) e^{-s z} d z = \sum_{n=0}^{\infty} a_{n} \frac{n!}{s^{n+1}}.
\end{align*}
The formal Borel transform is essentially equivalent to the formal inverse Laplace transform, more precisely, $s L(B(h))(s) = h\left(s^{-1}\right)$ holds.

\begin{prop}[Rogers]
We assume that the formal power series $h(x)=\sum_{n=0}^{\infty} a_{n} x^{n}$ can be expressed in at least the first of the following two continued fractions (here we assume that the coefficients $a_{n}, \alpha_{n}, \beta_{n}$ and $e_{n}$ do not depend on $x$):
\begin{align*}
\cfrac{a_{0}}{1-\alpha_{1}x-
\cfrac{\beta_{1}x^{2}}{1-\alpha_{2}x-
\cfrac{\beta_{2}x^{2}}{1-\alpha_{3}x- \cdots }}}, \quad  
%%%%%%%%%%
\cfrac{a_{0}}{1-
\cfrac{e_{1}x}{1-
\cfrac{e_{2}x}{1- \cdots }}} .
\end{align*}
If the above two expansions are possible, then we have $\alpha_{1} = e_{1}, \, \alpha_{n} = e_{2n-2}+e_{2n-1}\, (n\ge2)$ and $\beta_{n}=e_{2n-1}e_{2n}\, (n\ge1)$. 
(We omit the relation that the coefficients $a_{n}$ and $e_{n}$ satisfy. See \cite[\S 1]{rogers1907representation} for more details.) 
Using these numbers $\alpha_{n}$ and $\beta_{n}$, we define the sequence of formal power series $\{ \phi_{n}(x) \}_{n\ge0}$ and $\{h_{r}(x)\}_{r\ge1}$ as follows\footnote{Considering the consistency with the continued fraction (4) in \cite[\S 5]{rogers1907representation}, this recurrence formula is correct, and the sign of $\beta_{n}$ in equation (3) in the same section is incorrect.}.
\begin{align*}
&{} \phi_{0}(x) = h(x), \quad (1-\alpha_{1} x) \phi_{0}(x) = a_{0} + \beta_{1} x^{2} \phi_{1}(x), \\
&{} \phi_{n-2}(x) = (1-\alpha_{n} x) \phi_{n-1}(x) - \beta_{n} x^{2} \phi_{n}(x) \; (n\ge2), \quad h_{r}(x) = B(t^{r}\phi_{r}(t))(x).
\end{align*}
Then the following ``addition formula'' holds:
\begin{align}
B(h)(x+y) = A_{0} B(h)(x)B(h)(y) + \sum_{r=1}^{\infty} A_{r} h_{r}(x)h_{r}(y),  \label{eq:additionformula}
\end{align}
where $A_{0}=1/a_{0}$ and $A_{r+1}=\beta_{r+1}A_{r} \, (r\ge0)$\footnote{By comparing equations (2) and (3) (with the sign of $\beta_{n}$ changed to minus) in \cite[\S 5]{rogers1907representation}, we have $\frac{d}{d x} f_{n} =  f_{n-1}(x) + \alpha_{n+1} f_{n}(x) + \beta_{n+1} f_{n+1}(x)$ and hence $A_{n+1}=\beta_{n+1}A_{n}$ is correct, not $A_{n+1}=\beta_{n}A_{n}$.}. 
Conversely, if the addition formula is known, that is, if we know all of $B(h)(x), A_{r}$, and $h_{r}(x)$, we can convert the corresponding formal power series $h(x)$ into a continued fraction. 
\end{prop}
Let us consider a simple example using the above proposition.
\begin{ex}
We set $h(x)= \frac{1}{1+x^{2}}=\sum_{n=0}^{\infty}(-1)^{n}x^{2n}$ and then 
\begin{align*}
a_{2n}=(-1)^{n},\, a_{2n+1}=0 \,(n\ge0),\, \alpha_{n}=0 \, (n\ge1),\, \beta_{1} = -1, \, \beta_{n}=0 \,(n\ge2).
\end{align*}
Therefore, we have $\phi_{0}(x)=\phi_{1}(x)=\frac{1}{1+x^{2}}$ and
\begin{align*}
&{} B(h)(x) = \sum_{n=0}^{\infty} (-1)^{n} \frac{x^{2n}}{(2n)!} = \cos (x), \\
&{} h_{1}(x) = B( t \phi_{1}(t) )(x) = \sum_{n=0}^{\infty} (-1)^{n} \frac{x^{2n+1}}{(2n+1)!} = \sin (x).
\end{align*}
Since $A_{0}=1,\, A_{1}=-1, \, A_{r}=0\, (r\ge2)$, substituting these into  \eqref{eq:additionformula} gives the addition formula for the cosine function  $\cos (x+y) = \cos (x) \cos (y) - \sin (x) \sin (y)$ (as a formal power series). 
\end{ex}

\begin{proof}[``Hypergeometric'' proof of Theorem \ref{thm:AtkinFaberFaber}]
First, we recall the moment-generating function of the Atkin inner product \eqref{eq:Stieltjes function} and set 
\begin{align}
h(x) = \sum_{n=0}^{\infty} \mathcal{L}(j^{n}) x^{n} = \frac{{}_{2}F_{1} \left( \frac{5}{12}, \frac{13}{12} ; 1 ; 1728 x \right) }{ {}_{2}F_{1} \left( \frac{1}{12}, \frac{5}{12} ; 1 ; 1728x \right) }. \label{eq:h2F1}
\end{align}
From Equation (19) in \cite{kaneko1998supersingular}, we have
\begin{align*}
&{} e_{1} = 720, \; e_{n} = 12 \left( 6 + \frac{(-1)^{n}}{n-1} \right) \left( 6 + \frac{(-1)^{n}}{n} \right) \; (n\ge2), \\
&{} \alpha_{1} = 720, \; \alpha_{n} = \frac{24(144(n-1)^{2}-29)}{(2n-1)(2n-3)} \; (n\ge2), \\
&{}  \beta_{1} = 393120, \; \beta_{n} = \frac{36(12n-13)(12n-7)(12n-5)(12n+1)}{n(n-1)(2n-1)^{2}} \; (n\ge2), \\
&{} A_{0} =1=N_{0,2}, \; A_{r}= \beta_{1} \beta_{2} \cdots \beta_{r} = N_{r,2} \; (r\ge1),
\end{align*}
and then, from a contiguous relation of a hypergeometric series, we finally obtain 
\begin{align*}
\phi_{r}(x) = \frac{{}_{2}F_{1} \left( r +\frac{5}{12}, r +\frac{13}{12} ; 2r+1 ; 1728 x \right) }{ {}_{2}F_{1} \left( \frac{1}{12}, \frac{5}{12} ; 1 ; 1728x \right) } \quad (r\ge0).
\end{align*}
Noting that the formal Borel transform
\begin{align*}
B(h)(x+y) = \sum_{n=0}^{\infty} \mathcal{L}(j^{n}) \frac{(x+y)^{n}}{n!} = \sum_{n=0}^{\infty} \mathcal{L}(j^{n}) \sum_{k=0}^{n} \frac{x^{n-k}}{(n-k)!} \frac{y^{k}}{k!}
\end{align*}
and letting the formal Laplace transforms $j(p) L(*(x))(j(p))$ and $j(q) L(*(y))(j(q))$ act on both sides of \eqref{eq:additionformula} with $h(x)$ as \eqref{eq:h2F1}, we have
\begin{align}
\begin{split}
&{} \sum_{n=0}^{\infty} \mathcal{L}(j^{n}) \sum_{k=0}^{n} j(p)^{-(n-k)} j(q)^{-k} \\
&= \phi_{0} (j(p)^{-1}) \phi_{0} (j(q)^{-1}) + \sum_{r=1}^{\infty} N_{r,2}\, j(p)^{-r}\phi_{r} (j(p)^{-1}) j(q)^{-r} \phi_{r} (j(q)^{-1}) .
\end{split} \label{eq:addphi}
\end{align}
The left-hand side of \eqref{eq:addphi} is equal to
\begin{align*}
\sum_{n=0}^{\infty} \mathcal{L}(j^{n}) \frac{j(p)^{-n-1} - j(q)^{-n-1}}{j(p)^{-1} - j(q)^{-1}} = \frac{1}{{j(p)^{-1} - j(q)^{-1}}} \left( \frac{E_{2}(p)E_{4}(p)}{j(p) E_{6}(p)} - \frac{E_{2}(q)E_{4}(q)}{j(q) E_{6}(q)} \right).
\end{align*}
On the other hand, since
\begin{align*}
j(p)^{-r} \phi_{r}(j(p)^{-1}) &= \frac{{}_{2}F_{1} \left( r +\frac{5}{12}, r +\frac{13}{12} ; 2r+1 ; \tfrac{1728}{j(p)} \right) }{ j(p)^{r} {}_{2}F_{1} \left( \frac{1}{12}, \frac{5}{12} ; 1 ; \tfrac{1728}{j(p)} \right) } \\
&= \left( 1- \frac{1728}{j(p)} \right)^{-1/2} \frac{{}_{2}F_{1} \left( r -\frac{1}{12}, r +\frac{7}{12} ; 2r+1 ; \tfrac{1728}{j(p)} \right) }{ j(p)^{r} {}_{2}F_{1} \left( \frac{1}{12}, \frac{5}{12} ; 1 ; \tfrac{1728}{j(p)} \right) } \\
&= \frac{E_{4}(p) G_{12r+2}(p) }{E_{6}(p) \Delta(p)^{r}} \quad (\text{by Proposition \ref{prop:Gto2F1}}),
\end{align*}
the right-hand side of \eqref{eq:addphi} is equal to 
\begin{align*}
\frac{E_{4}(p)E_{4}(q)}{E_{6}(p)E_{6}(q)} \sum_{r=0}^{\infty} N_{r,2} \, G_{12r+2}(p) G_{12r+2}(q) \left( \Delta(p)\Delta(q) \right)^{-r} .
\end{align*}
After a short calculation we obtain the desired result. 
\end{proof}

%%%%%%%%%%%%%%%%%%%%%%%%%%%%%%%%%%%%%%%%%%%%%%%%%%%%%%%%%%%%%%%%%%%%%%%%%%%%%%%%%%%%%%%%%
%%%%%%%%%%%%%%%%%%%%%%%%%%%%%%%%%%%%%%%%%%%%%%%%%%%%%%%%%%%%%%%%%%%%%%%%%%%%%%%%%%%%%%%%%
%%%%%%%%%%%%%%%%%%%%%%%%%%%%%%%%%%%%%%%%%%%%%%%%%%%%%%%%%%%%%%%%%%%%%%%%%%%%%%%%%%%%%%%%%

\section{Hypergeometric aspects of the generalized Faber polynomials} \label{sec:HypgenFaber}
In this short section, we consider the hypergeometric aspects of the generalized Faber polynomials. Of course, it is not a hypergeometric polynomial, but it has an interesting connection with the hypergeometric series.

Assuming $\alpha, \beta \not\in \mathbb{Z}_{<0}$ and setting 
\begin{align*}
{}_{2}G_{1} \left( \alpha, \beta ; 1 ; z \right) \coloneqq \sum_{n=1}^{\infty} \frac{(\alpha)_{n}(\beta)_{n}}{n!^{2}} \left\{ \sum_{k=0}^{n-1} \left( \frac{1}{\alpha +k} + \frac{1}{\beta +k} -  \frac{2}{1+k} \right)  \right\} z^{n}.
\end{align*}
Then the two-dimensional solution space of the hypergeometric differential equation $\{\Theta^{2} -z (\Theta +\alpha)(\Theta +\beta)\} F =0$ is spanned by ${}_{2}F_{1} \left( \alpha, \beta ; 1 ; z \right)$ and ${}_{2}G_{1} \left( \alpha, \beta ; 1 ; z \right) + \log (z) {}_{2}F_{1} \left( \alpha, \beta ; 1 ; z \right)$. 
It is classically well known that, roughly speaking, the inverse map of a modular function of certain triangle groups is given by the Schwartz map, which is the ratio of the independent solutions of a hypergeometric differential equation.

Now we consider the modular function $t=1/j(q)=1/j(\tau) : \Gamma \backslash \mathfrak{H} \ni \tau \mapsto t \in \mathbb{C} \cup \{ \infty \}$. The first few terms of its Fourier expansion are given by
\begin{align}
t = q -744 q^{2} +356652 q^{3} -140361152 q^{4} +49336682190 q^{5}  +O(q^{6}), \label{eq:reciprocalj} %A066395 
\end{align}
and $2\pi i  \tau = \log(q)$ is given by the ratio of two functions
\begin{align*}
{}_{2}F_{1} \left( \frac{1}{12}, \frac{5}{12} ; 1 ; 1728t \right) \; \text{ and } \; {}_{2}G_{1} \left( \frac{1}{12}, \frac{5}{12} ; 1 ; 1728t \right) + \log (t)\,  {}_{2}F_{1} \left( \frac{1}{12}, \frac{5}{12} ; 1 ; 1728t \right).
\end{align*}
We note that essentially the same of these equations can be found in 
\cite[Ab.~I, \S 9]{klein1879transformation} by Klein. 
Thereby, the inverse series $q=q(t)$ can be expressed as follows:
\begin{align}
q &= t \cdot \exp \left(  \frac{{}_{2}G_{1} \left( \frac{1}{12}, \frac{5}{12} ; 1 ; 1728t \right) }{ {}_{2}F_{1} \left( \frac{1}{12}, \frac{5}{12} ; 1 ; 1728t \right) }  \right) \label{eq:normq} \\
&=  t+744 t^{2}+750420 t^{3}+872769632 t^{4}+1102652742882 t^{5} +O(t^{6}).  \notag %OEIS(A091406)
\end{align}
See also \cite[p.~336, Eq.~(10)]{fricke1916elliptischen} by Fricke, \cite[p.~21]{higher1955erdelyi}, \cite[\S 4]{lian1996arithmetic}.

\begin{rem}
Note that the integrality of the coefficients in \eqref{eq:reciprocalj} and  \eqref{eq:normq} is  equivalent. The series \eqref{eq:reciprocalj} is a mirror map of a certain family of $K3$ surfaces, and for more details of this fact, we refer to \cite[\S 4--\S 5]{lian1996arithmetic}. The integrality of the Fourier coefficients in \eqref{eq:reciprocalj}, and therefore the coefficients in  \eqref{eq:normq}, can be easily shown, but it is generally unclear whether a mirror map determined from the solution of a Picard-Fuchs differential equation has such a property. However, the integrality is known to hold for a fairly wide class, and for interested readers we refer to \cite{delaygue2017hypergeometricmirrormaps} and its references.
\end{rem}

For the power series $f(t)=\sum_{n=0}^{\infty} a_{n} t^{n}$, we denote the symbol $\langle f(t) \rangle_{\ell}$ as $\sum_{n=0}^{\ell} a_{n} t^{n}$. 
We can obtain the generalized Faber polynomials directly, without computing the Fourier expansion of equations \eqref{eq:genfct2-k} and \eqref{eq:genfctk}, by the following assertion.
\begin{prop}\label{prop:HypgenFaberpoly}
Put $t=1/j$, $\mathcal{F}_{1}(t)={}_{2}F_{1}\left( \frac{1}{12}, \frac{5}{12} ; 1 ; 1728t \right)$ and $\mathcal{G}_{1}(t)={}_{2}G_{1}\left( \frac{1}{12}, \frac{5}{12} ; 1 ; 1728t \right)$. 
For any even integer $k=12m+4\delta+6\varepsilon$ with $m \in \mathbb{Z}, \, \delta \in \{0,1,2\}, \, \varepsilon \in \{0,1\}$, and an integer $\ell \ge -m$, we have 
\begin{align*}
F_{k,\ell+m}(j) = j^{\ell+m} \left\langle \left( 1 - 1728 t \right)^{-\varepsilon/2} \mathcal{F}_{1}(t)^{-k} \exp \left( - \ell \, \frac{ \mathcal{G}_{1}(t) }{ \mathcal{F}_{1}(t) }  \right) \right\rangle_{\ell+m} .
\end{align*}
\end{prop}

\begin{proof}
Substituting \eqref{eq:normq} for $q^{-\ell}$ in \eqref{eq:basefkl}, we have
\begin{align*}
f_{k,\ell}(q) = q^{-\ell}(1+O(q^{\ell+m+1})) = t^{-\ell} \exp \left( - \ell \, \frac{ \mathcal{G}_{1}(t) }{ \mathcal{F}_{1}(t) }  \right) (1+ O(t^{\ell+m+1})).
\end{align*}
Then, by rewriting the definition of the generalized Faber polynomial using the hypergeometric expressions of the Eisenstein series in Proposition \ref{prop:EisensteinHyp}, we obtain
\begin{align*}
j^{-\ell-m} F_{k,\ell+m}(j) &= t^{\ell+m} E_{4}^{-\delta}E_{6}^{-\varepsilon} \Delta^{-m} f_{k,\ell}(q) = \left( 1 - 1728 t \right)^{-\varepsilon/2} \mathcal{F}_{1}(t)^{-k}\, t^{\ell} f_{k,\ell}(q) \\
&= \left( 1 - 1728 t \right)^{-\varepsilon/2} \mathcal{F}_{1}(t)^{-k} \exp \left( - \ell \, \frac{ \mathcal{G}_{1}(t) }{ \mathcal{F}_{1}(t) }  \right) (1+ O(t^{\ell+m+1})).
\end{align*}
Since the left-hand side is a polynomial of  degree $\ell+m$ with respect to $t$, we obtain the desired result by truncating the series on the right-hand side at order $\ell+m$. 
\end{proof}

Using Proposition \ref{prop:HypgenFaberpoly}, we can explicitly compute some higher-order coefficients of the generalized Faber polynomials. For instance, setting $F_{k,\ell+m}(X) = X^{\ell+m} + \sum_{i=1}^{\ell+m} c_{i}(k,\ell) X^{\ell+m-i}$ and then 
\begin{align*}
c_{1}(k,\ell) &= -744 \ell -12(5k-72\varepsilon), \\
c_{2}(k,\ell) &= 26768 \ell^{2} + 36 (1240 k - 17856 \varepsilon -13157) \ell \\
&{} \quad + 36 \left\{ 50 k^{2} - 5(288\varepsilon +211)k   +31104  \varepsilon \right\} . 
\end{align*}

\section*{Acknowledgement}
The author would like to thank Professor Masanobu Kaneko and Dr.~Yuichi Sakai for their helpful comments on a draft of this paper.

%%%%%%%%%%%%%%%%%%%%%%%%%%%%%%%%%%%%%%%%%%%%%%%%%%%%%%%%%%%%%%%%%%%%%%%%%%%%%%%%%%%%%%%%%
%%%%%%%%%%%%%%%%%%%%%%%%%%%%%%%%%%%%%%%%%%%%%%%%%%%%%%%%%%%%%%%%%%%%%%%%%%%%%%%%%%%%%%%%%
%%%%%%%%%%%%%%%%%%%%%%%%%%%%%%%%%%%%%%%%%%%%%%%%%%%%%%%%%%%%%%%%%%%%%%%%%%%%%%%%%%%%%%%%%

\bibliographystyle{abbrv}
\bibliography{ref-nakaya2023orthogonalityAtkin-likepoly.bib}

%%%%%%%%%%%%%%%%%%%%%%%%%%%%%%%%%%%%%%%%%%%%%%%%%%%%%%%%%%%%%%%%%%%%%%%%%%%%%%%%%%%%%%%%%

\end{document}